\documentclass[a4paper,reqno]{amsart}

\usepackage{amsmath,amsthm,amssymb,amsfonts,amsbsy}
\usepackage{mathtools}      
\usepackage{mathabx}        
\usepackage{enumitem,color,graphicx,cite}
\usepackage[all,pdf]{xy}

\usepackage{geometry}  
\usepackage[backref=page,hyperindex=true,CJKbookmarks=true,
colorlinks,linkcolor=blue,anchorcolor=red,citecolor=cyan]{hyperref}

\theoremstyle{plain}
\newtheorem{thm}{Theorem}[section]

\newtheorem{prop}[thm]{Proposition}
\newtheorem{lem}[thm]{Lemma}

\theoremstyle{definition}
\newtheorem{defn}[thm]{Definition}

\newtheorem{claim}[thm]{Claim}

\theoremstyle{remark}
\newtheorem{rmk}[thm]{Remark}

\theoremstyle{plain}
\newtheorem{maintheorem}{Theorem}

\newtheorem{main-cor}[maintheorem]{Corollary}
\newtheorem{main-prop}[maintheorem]{Proposition}
\newtheorem*{notation*}{Notation}
\numberwithin{equation}{section}

\newcommand{\mcl}{\mathcal{L}}
\newcommand{\mcf}{\mathcal{F}}

\newcommand{\mtt}{\mathbb{T}^2\to\mathbb{T}^2}
\newcommand{\mrr}{\mathbb{R}^2\to\mathbb{R}^2}

\newcommand{\mttd}{\mathbb{T}^d\to\mathbb{T}^d}

\def\NN{\mathbb{N}}
\def\RR{\mathbb{R}}

\def\TT{\mathbb{T}}
\def\ZZ{\mathbb{Z}}

\def\tildeL{\widetilde{\mathcal{L}}}

\def\e{{\varepsilon}}

\def\mathcalf{\mathcal{F}_f}

\renewcommand*{\backref}[1]{}
\renewcommand*{\backrefalt}[4]{%
	\ifcase #1 (Not cited.)%
	\or        (Cited on page~#2.)%
	\else      (Cited on pages~#2.)%
	\fi}

\address[R. Gu]{School of Mathematical Sciences, 
	Tongji University, 
	Shanghai 200092, P.R. China}
\email{rhgu@tongji.edu.cn}

\address[M. Xia]{School of Mathematical Sciences, 
	Dalian University of Technology, 
	Dalian 116024, P.R. China}
\email{xiamingyang@dlut.edu.cn}

\title[DA local diffeomorphisms]
{On the Absence of Anosov Factors for DA Local Diffeomorphisms}
\author[R. Gu and M. Xia]{Ruihao GU and Mingyang XIA}

\date{\today}

\begin{document}

\begin{abstract}
	We give a class of local diffeomorphisms 
	which are homotopic to toral hyperbolic endomorphisms,
	but which are not topologically semi‑conjugate, 
	within the homotopic class of the identity,	to any Anosov local diffeomorphism.
	In particular, we show that 
	for a non-invertible partially hyperbolic $C^1$-smooth local diffeomorphism with expanding directions on the $2$-torus, 
	if it is semi-conjugate to an Anosov local diffeomorphism, then it is also an Anosov local diffeomorphism.
\end{abstract}

\maketitle

\section{Introduction}

Let $f,g:M\to M$ be two local diffeomorphisms of a $C^{\infty}$-smooth closed Riemannian manifold $M$.
We say that $f$ is topologically \textit{semi-conjugate }to $g$ 
if there exists a continuous surjection $h:M\to M$ homotopic to the identity Id$_M$, 
such that $h\circ f=g\circ h$. 
In this case, we call $g$ a \textit{factor} of $f$. 
Moreover, if such a \textit{semi-conjugacy} $h$ is homeomorphic, 
then $h$ is called a topological \textit{conjugacy}. 

In 1967,  Smale \cite{Smale1967} conjectured, up to topological conjugacy, 
a classification of completely hyperbolic systems for Anosov diffeomorphisms and expanding local diffeomorphisms. 
The classification of expanding local diffeomorphisms achieves \cite{F70,Gro81,Shub1969}: 
an expanding local diffeomorphism is topologically conjugate to an affine expanding endomorphism of an infra-nilmanifold. 
However, the classification of Anosov diffeomorphisms is widely open. 
The facts we know so far are the following:
If an Anosov diffeomorphism of $d$-manifold $M$ is codimension-one,
then $M$ is homeomorphic to the $d$-torus $\TT^d $\cite{F70,Newhouse1970};
If restricted to a nilmanifold, 
an Anosov diffeomorphism is topologically conjugate to an affine hyperbolic automorphism \cite{Franks1969,F70,Manning1974}. 

Later, beyond the hyperbolicity, Franks \cite{F70} and Shub \cite{Shub1969} generally 
explored to classify, up to topological semi-conjugacy, the homotopic classes of the completely hyperbolic systems.
\begin{thm}[\!\cite{F70,Shub1969}\,]\label{1 thm semi-conj holds}
	Let $f:\mttd$ be a diffeomorphism (resp. local diffeomorphism) 
	and $g:\mttd$ be an Anosov diffeomorphism (resp. expanding local diffeomorphism). 
	If $f$ is homotopic to $g$, then $f$ is topologically semi-conjugate to $g$.
\end{thm}

In 1974, Ma\~n\'e and Pugh \cite{ManePugh1975} extended the definitions of Anosov diffeomorphisms and expanding local diffeomorphisms to the following \textit{Anosov local diffeomorphisms}. 
\begin{defn}[\!\cite{ManePugh1975}\,]\label{1 def Anosov map manepugh} 
	Let $f:M\to M$ be a local diffeomorphism. We say $f$ is \textit{Anosov}, 
	if there exists a continuous $Df$-invariant subbundle $E_f^s$ of the tangent bundle $TM$ such that 
	$Df$ uniformly contracts on $E_f^s$ and uniformly expands on the quotient bundle $TM/E_f^s$.
\end{defn}

This unifies the cases of 
Anosov diffeomorphisms ($f$ is invertible) and expanding local diffeomorphisms ($E_f^s$ is trivial).  
Then, we call a local diffeomorphism $f:\mttd$ \textit{derived from Anosov} (abbr. \textit{DA}), 
if it is homotopic to a linear Anosov map $A=f_*$, the \textit{linearization} of $f$. 

Throughout this paper, 
by \emph{map} we mean a local diffeomorphism,
by \emph{Anosov map} we mean a non-invertible Anosov local diffeomorphism, 
which is neither an Anosov diffeomorphism nor an expanding map.

In the present work, we continue to study the homotopic class for toral Anosov maps, 
and show that the classification in Theorem \ref{1 thm semi-conj holds} fails for toral Anosov maps. 
Specifically, we have the following result.

\begin{maintheorem}\label{main-thm-factor}
	There exists a $C^1$-open subset $\mathcal{U}$ in the space of maps on $\TT^2$ 
	such that for every $f\in \mathcal{U}$, the map $f$ is derived from Anosov, 
	but $f$ is not topologically semi-conjugate to any Anosov map. 
	Indeed, for every linear Anosov map $A$ on $\TT^2$, there exists such a subset $\mathcal{U}$ 
	within the homotopic class of $A$, in $C^1$ topology.
\end{maintheorem}

\begin{rmk}\label{1 rmk special case of factor}
	In 2023, J\'{e}r\^{o}me Buzzi asked,
	at the conference of Beyond Uniform Hyperbolicity in Poland,
	whether Theorem \ref{1 thm semi-conj holds} holds for the case of Anosov maps.  
	Motivated by this, the authors showed 
	in \cite{GX2023} that 
	there exist local diffeomorphisms derived from Anosov on $\TT^2$ but not semi-conjugate to their linearizations. 
	This corresponds to the case of $g=f_*$ in Theorem \ref{1 thm semi-conj holds}.  
	However, it is not sufficient for a complete answer to above question, 
	nor for establishing Theorem \ref{main-thm-factor}. 
	Indeed, Ma\~n\'e-Pugh \cite{ManePugh1975} and Przytycki \cite{Pr76} proved that 
	any Anosov map is \textit{not structurally stable}, in particular, 
	there exist Anosov maps being $C^{1}$-close but not conjugate to the linear Anosov map $A$ in Theorem \ref{main-thm-factor}. 
	Note that generic Anosov maps on $\TT^2$ cannot be conjugate to its linearization \cite{AGGS23}. 
	Hence, Theorem \ref{main-thm-factor} needs to deal with much more general cases.
\end{rmk}

In this paper, we mainly consider DA maps with \textit{partially hyperbolic} splittings. 
In fact, we focus on a partially hyperbolic map $f:M\to M$ with uniformly expanding directions, 
i.e., there is a $Df$-invariant subbundle $E^c$ on $TM$ such that $Df$ uniformly expands on the quotient bundle $TM/E^c$ (see Definition \ref{2 def ph endo}).  
In the following, by \emph{DA map} we mean a non-invertible local diffeomorphism 
which is homotopic to a linear Anosov map and admits a partially hyperbolic splitting with uniformly expanding directions. 

To prove Theorem \ref{main-thm-factor}, we indeed show that 
Theorem \ref{1 thm semi-conj holds} fails to DA maps by a rigidity argument.  
The rigidity phenomenon on (partially) hyperbolic maps 
in the sense of the obstruction of a conjugacy having higher regularity has been extensively studied 
in decades \cite{deMM1986,dL92,G08,G17}. 
There are remarkable works appearing in recent years, such as 
the rigidity of Lyapunov exponents \cite{SY19}, 
the rigidity of SRB measures \cite{ALOS26}, 
the rigidity of invariant foliations \cite{GoS23,GKS23}, 
the global rigidity \cite{DG24,KSW24}, 
and the rigidity for very non-algebraic expanding maps \cite{GRH23}.  
Here, we consider the following rigidity of the existence of a semi-conjugacy 
between a DA map and an Anosov map. 
By $C^{r}\ (r>1)$, we mean that the regularity is at least $C^{1+\alpha}$ for some $0<\alpha<1$.

\begin{maintheorem}\label{main-thm-conjugacy}
	Let $f:\mtt$ be a $C^1$-smooth  DA map and $g:\mtt$ be a $C^{1}$-smooth Anosov map.  
	Assume that $f$ is topologically semi-conjugate to $g$ via  $h:\mtt$. 
	Then, $f$ is in fact an Anosov map, and $h$ is a homeomorphism and hence a topological conjugacy.
	Moreover, if we further assume that both $f$ and $g$ are $C^{r}$-smooth $(r>1)$, 
	then the homeomorphism $h$ is $C^r$-smooth along every leaf of the stable foliation.
\end{maintheorem}

\begin{rmk}\label{1 rmk compare to endo}
		Similar rigidity phenomena have been investigated for various settings in \cite{AGGS23,GS22,GX2023}.
		Here, we clarify the relationship between Theorem \ref{main-thm-conjugacy} and the corresponding results in these works.
		 
		Let $f$ and $g$ be local diffeomorphisms given in Theorem \ref{main-thm-conjugacy}, i.e., $f$ is DA and $g$ is Anosov.
			\begin{itemize}
				\item  In \cite{AGGS23}, the authors assume further that $f$ is also Anosov and $g=f_*$, 
				and show in $C^1$-regularity that the existence of a conjugacy $h$ implies that $f$ has constant periodic contraction. 
				In $C^{1+\alpha}$-regularity, the conjugacy $h$ is automatically  $C^{1+\alpha}$-smooth along  stable manifolds. 
				\item In \cite{GS22}, the authors consider that both $f$ and $g$ are Anosov maps, 
				and show also in $C^{1+\alpha}$-regularity that the conjugacy $h$ is automatically $C^{1+\alpha}$-smooth along  stable manifolds. 
				\item In \cite{GX2023}, the authors prove Theorem \ref{main-thm-conjugacy} 
				assuming $g = f_*$ to be the linearization of $f$. 
			\end{itemize}
		Therefore, \cite{GS22} extends \cite{AGGS23} to general Anosov case in the $C^{1+\alpha}$ setting, 
		as we have seen in Remark \ref{1 rmk special case of factor} that Anosov maps have no structural stability.  
		On the one hand, \cite{GX2023} can be regarded as a DA-version continuation of \cite{AGGS23}.  
		On the other hand, the present work can be regarded a DA-version continuation of \cite{GS22}. 
		Furthermore, it is worth pointing out that 
		there exists no such rigidity result of Theorem \ref{main-thm-conjugacy} in \cite{GS22} 
		under the $C^1$ assumption. 
		Thus, the rigidity part of Theorem \ref{main-thm-conjugacy} in the $C^1$ setting is also a novelty of this paper.
\end{rmk}

\begin{rmk}\label{1 rmk compare to T3}
	{\rm Notice that DA local diffeomorphisms on $\TT^2$ can be seen as analogues of DA diffeomorphisms on $\TT^3$ 
		(see also \cite{AGGS23,GS22} for some discussion).
		We also would like to compare Theorem \ref{main-thm-conjugacy} to rigidity results of 
		DA diffeomorphism on $\TT^3$ \cite{GS20,HS21} in which the $C^{1+\alpha}$ assumption is necessary. 
		To clarify the relationship between our work and the results in \cite{GS20,HS21}, 
		we will provide a detailed discussion in Subsection \ref{subsec DA diffeo} after introducing some preliminary notions.}
\end{rmk}

We would like to explain one of the $C^1$-techniques in the present work.
To go beyond the situation of \cite{GX2023} as discussed in Remark \ref{1 rmk compare to endo},
we need to deal with the case of $g$ being non-linear and show that $f$ is Anosov by the following process:
\begin{enumerate}[label=(\roman*)]
	\item the semi-conjugacy is a conjugacy; 
	\item the conjugacy is bi-H\"older  continuous; 
	\item the H\"older conjugacy implies the hyperbolicity of $f$. 
\end{enumerate}
It is worth pointing out that Fisher in PhD thesis  \cite{F2004} also gives a result that 
a $C^2$ diffeomorphism bi-H\"older (with large H\"older exponents) conjugate to an Anosov diffeomorphism 
is also Anosov.
In \cite{G10}, Gogolev gave a counter-example: 
a diffeomorphism is not Anosov, but is H\"older conjugate to an Anosov diffeomorphism. 
Note that the example in \cite{G10} admits no globally partially hyperbolic splitting, 
while the DA map $f$ in Theorem \ref{main-thm-conjugacy} does. 
To the best of our knowledge, 
no such argument to establish the hyperbolicity of $f$ has appeared in the literature. 
Based on the work of \cite{GS20,HS21} as mentioned in Remark \ref{1 rmk compare to T3},  
we also apply this argument to show a new rigidity result for DA diffeomorphism on $\TT^3$;
see Subsection \ref{subsec DA diffeo} and Appendix \ref{sec-app-c}.

Partially hyperbolic DA systems have been attractive subjects, 
since Ma\~{n}\'{e} \cite{M78} constructed the robustly transitive but non-hyperbolic examples on $\TT^3$. 
The partial hyperbolicity and the semi-conjugacy (Theorem \ref{1 thm semi-conj holds}) 
preserve many dynamical tools and properties of uniformly hyperbolic systems, 
such as nice foliation structures \cite{FPS2014,Potrie2015,HUY22} and good statistics properties \cite{U12,BFSV12,PTV2018}.
However, as shown by Theorem \ref{main-thm-factor}, DA maps may fail to admit Anosov factors.
Hence, one cannot expect to directly obtain dynamical information of a non-hyperbolic DA map
from its linearization — even in the setting of Anosov systems — via a non‑existent semi‑conjugacy on the ambient manifold.
Instead, one may study semi‑conjugacies on the universal covers and on the orbit spaces,
though they may not descend to the base manifold.
For instance, structural stability still holds at the level of orbit spaces; see \cite{AH94,BR13,BK17}.

It is worth mentioning that statistical properties of local diffeomorphisms have been investigated in two complementary settings. Andersson, Carrasco and Saghin \cite{ACS25} constructed $C^1$-open families of conservative, stably ergodic, 
non-uniformly hyperbolic DA local diffeomorphisms on $\TT^2$ that lack dominated splittings. 
We also refer to \cite{BCF18} for the study of ``bad'' entropy behaviors of $C^1$-smooth diffeomorphisms without dominated splittings.
In a different direction, 
Tsujii proved that for sufficiently smooth partially hyperbolic local surface diffeomorphisms with uniformly expanding directions, 
generic maps admit finitely many ergodic physical measures whose union of basins has full Lebesgue measure \cite{T05}. 
From the perspective of SRB measures, Theorem \ref{main-thm-conjugacy} yields a rigidity result for Gibbs states 
within the class of dominated DA systems.

\vspace{8pt}\noindent
\textbf{Organization of this paper.}
The rest of the paper is organized as follows.
In Section \ref{sec Pre}, we recall some basic properties of DA maps.
In Section \ref{sec: dichotomy}, we establish a dichotomy concerning accessibility for DA maps,
which allows us to split the proof of Theorem \ref{main-thm-conjugacy} into two cases.
The accessible case of Theorem \ref{main-thm-conjugacy} is proved in Section \ref{sec rigidty},
while the special case was treated mainly in our previous work \cite{GX2023}.
In Section \ref{sec DA without Anosov factors}, we prove Theorem \ref{main-thm-factor}.
We conclude with three appendices.
In Appendix \ref{sec-app-a}, 
we present a technical intermediate step for the above‑mentioned accessibility dichotomy,
which is also of independent interest.
In Appendix \ref{sec-app-b}, 
we establish some rigidity results for partially hyperbolic DA local diffeomorphisms admitting stable bundles.
In Appendix \ref{sec-app-c}, 
by applying our $C^1$-argument to DA diffeomorphisms on $\mathbb T^3$,
we obtain a new rigidity result for the diffeomorphism case.

\section{Preliminaries}\label{sec Pre}

\subsection{Partially hyperbolic local diffeomorphism}\label{sec: phe}

One can define partial hyperbolicity for local diffeomorphism by considering the quotient bundle 
as in Definition \ref{1 def Anosov map manepugh}.
Here, we give the equivalent definition on the inverse limit space. 
Then we observe the dynamics on the universal cover.  
We refer to \cite{GX2023} for the equivalence of definitions, involving the cone-fields.

\subsubsection{\bf On the inverse limit space}

Let $d(\cdot,\cdot)$ be a metric of $\TT^2$,
and $(\TT^2)^{\mathbb{Z}}:=\{(x_i) \;|\;x_i\in \TT^2, \forall i\in\, \mathbb{Z} \}$ be the product topological space of $\TT^2$
which is compact and metrizable by the metric
\begin{align*}
	\bar{d}((x_i),(y_i))= \sum_{-\infty}^{+\infty} \frac{d(x_i,y_i)}{2^{|i|}}.
\end{align*}
For any $\bar{x}=(x_i)\in (\TT^2)^{\ZZ}$, 
let $(\bar{x})_i:=x_i$ and $\pi_i:(\TT^2)^{\ZZ}\to \TT^2$ be the projection $\pi_i\big(\bar{x}\big)=x_i$. 
Let $\sigma: (\TT^2)^{\mathbb{Z}}\to (\TT^2)^{\mathbb{Z}}$ be the \textit{shift} homeomorphism 
$\big(\sigma (\bar{x})\big)_j=x_{j+1}$, for all $j\in \mathbb{Z}$.
The \textit{inverse limit space} of $f$ is defined by
\begin{align}
	\TT^2_f:=\big\{\bar{x}=(x_i) \;|\; x_i\in \TT^2 \;\; {\rm and} \;\; f(x_i)=x_{i+1},  \forall i\in \mathbb{Z} \big\}.\label{eq. 2.1 inverse limit space}
\end{align}
Denote the restriction of $\sigma$ on $\TT^2_f$ by $\sigma_f$. 
The space $(\TT^2_f,\bar{d})$ is a $\sigma_f$-invariant compact metric space.

\begin{defn}\label{2 def ph endo} 
	A local diffeomorphism $f:\mtt$ is \textit{partially hyperbolic}, 
	if there are constants $C>0,\lambda>1$ and $k\in\NN$ such that for any orbit $\bar{x}=(x_i)\in\TT^2_f$, 
	there is a splitting  $T_{x_i}\TT^2=E^1_f(x_i, \bar{x})\oplus E^2_f(x_i, \bar{x})$  such that
	\begin{itemize}
		\item (Invariance) 
		$Df(E^j_f(x_i, \bar{x}))=E^j_f(x_{i+1}, \bar{x})$, 
		for $j=1,2$ and all $i\in\ZZ$;
		
		\item (Domination) 
		$\|D_{x_i}f^kv_1\|\leq \frac{1}{2}\|D_{x_i}f^kv_2\|$, 
		for all $i\in \ZZ$ and unit vectors $v^{1/2}\in E^{1/2}_f(x_i,\bar{x})$; 
		
		\item (Partial hyperbolicity)  
		For all $i\in\ZZ$ and unit vectors $v^{1/2}\in E^{1/2}_f(x_i,\bar{x})$,
		\begin{enumerate}[label=(\alph*)]
			\item either $\|D_{x_i}f^nv_2\|\geqslant C\lambda^{n} $,  for all $n\in \NN$;
			\item or 	 $\|D_{x_i}f^nv_1\|\leqslant C\lambda^{-n}$,  for all $n\in \NN$.
		\end{enumerate}
	\end{itemize}
	In particular, when both inequalities (a) and (b) hold, $f$ is an \emph{Anosov local diffeomorphism}.
\end{defn}

It is known that both bundles $E^{1/2}_f(x,\bar{x})$ are continuous with respect to $(x,\bar{x})\in \TT^2\times \TT^2_f$ 
and indeed H\"older continuous \cite{Pr76,Pesinbook04}. 
Note that $E^1_f(x_0,\bar{x})$ depends only on $x_0$, 
so we can denote it by $E^1_f(x_0)$, while $E^2_f(x_0,\bar{x})$ generally relies on the choice of negative orbits.  

If (a) of Definition \ref{2 def ph endo} holds,  
$f$ has a uniformly expanding bundle (with respect to orbits). 
Here we focus on the case (a) in the present paper. 
Then we call $f$ a \emph{DA map}, if its linearization is a linear Anosov map, 
and it is partially hyperbolic with a uniformly expanding bundle, denoted by 	
\[T_{x_i}\TT^2=E^c_f(x_i)\oplus E^u_f(x_i, \bar{x}),\]	
where $E^c_f$ and $E^u_f$ are the center bundle and unstable bundle of $f$, respectively.  

If (b) of Definition \ref{2 def ph endo} holds, $f$ admits a uniformly contracting bundle, 
called \emph{stable bundle} and denoted by $E^s_f$ which is unique integrable to the \emph{stable foliation}. 
For completeness, this case of partial hyperbolicity is also considered in Appendix \ref{sec-app-b}.   

Let $f$ be a DA map. By \cite{Pr76}, there exists submanifold $\mcf_f^{u} (x_0,\bar{x})$ of $\TT^2$ 
called \textit{unstable manifold}, tangent to $E^{u}_f(x_0,\bar{x})$ 
such that for any $y_0\in \mcf_f^{u} (x_0,\bar{x})$, there exists $\bar{y}=(y_i)\in \TT^2_f$ such that  
if $y_0\in \mcf_f^{u}(x_0,\bar{x})$, then $d(x_i,y_i)\to 0$ as $i\to -\infty$. 
Again, the unstable manifolds generally depend on the choice of negative orbits.

\subsubsection{\bf On the universal cover space}

Let $f$ be a DA map on $\TT^2$.
Let $F:\mrr$ be a lift of $f$, and $\pi:\RR^2\to \TT^2$ be the natural projection.  
Then $F$ is a partially hyperbolic diffeomorphism \cite{ManePugh1975} with a unstable bundle, 
namely, there exist a continuous $DF$-invariant splitting $T\RR^2=E^c_F\oplus E^u_F$ 
and constants $C>0,\lambda>1$ and $k\in\NN$ such that   
for any $x\in \RR^2$, unit vectors $v_{c/u}\in E^{c/u}_F(x)$, and any $n\in\NN$,
\[ 
\|D_xF^kv_c\|\leqslant \frac{1}{2} \|D_xF^kv_u\|
\quad {\rm and}\quad  
\|D_{x}F^nv_u\|\geqslant C\lambda^{n}. 
\]
We call $E^{c/u}_F$ the center/unstable bundle of $F$. 
Recall that $E^{u}_F$ are uniquely integrable to the so-called \emph{unstable foliation} which is denoted by $\mcf^{u}_F$. 
In the case that $f$ is a DA map, the center bundle $E^c_F$ is integrable \cite{HH21,GX2023} 
(see also Proposition \ref{2 prop foliation on R2}). 
We denote the integrable foliation by $\mcf_F^c$, and call it the \emph{center foliation}. 

\begin{rmk}\label{2 rmk universal to limit}
	The projection of the orbit space $\RR^2_F$ is dense in $\TT^2_f$ \cite{MT16}, 
	i.e., for every orbit $\bar{x}=(x_i)\in \TT^2_f$, 
	there exists $y_k\in\RR^2$ with $\pi(y_k)=x_0$ such that  
	$\bar{y}^k:=\big(\pi(F^i(y_k))\big)_{i\in\ZZ} \in \TT^2_f$ satisfies 
	$\bar{d}(\bar{y}^k, \bar{x})\to 0$ as $k\to +\infty$. 
	By the continuity of bundles, we have
	$$D\pi\big( E^{c/u}_F(y_k)\big) \to E^{c/u}_f(x_0, \bar{x})\ \text{as}\ k\to +\infty.$$ 
	It is clear that for every $x\in\RR^2$ and the corresponding orbits $\bar{x}=(x_i)_{i\in\ZZ}:=\big(\pi(F^i(x))\big)_{i\in\ZZ}\in \TT^2_f$, 
	the leaf $\pi\big( \mcf^{c/u}_F(x)\big)$ coincides with $\mcf_f^{c/u}(\pi(x),\bar{x})$. 
\end{rmk}

\subsubsection{\bf Special property and accessibility}

Let $f$ be a DA map on $\TT^2$. If the bundle $E^u_f$ at  $x$ is independent of the choice of negative orbits,  
we call $x$ a \textit{special point} of $f$. 
Moreover, we say that the map $f$ is \emph{special}, if every point $x\in \TT^2$ is special.
In this case, there exists a $Df$-invariant partially hyperbolic splitting 
\begin{align}
	T\TT^2= E_f^c\oplus_< E_f^u.\label{eq. 1. special}
\end{align}
By Remark \ref{2 rmk universal to limit}, 
$f$ is special if and only if 
the unstable bundle $E^u_F$ of $F$ is \textit{invariant under deck transformations}, i.e.,
for all $x\in \RR^2$ and all $n\in\ZZ^2$, 
$$D\mathcal{T}_n(E^{u}_F(x))=E^{u}_F(\mathcal{T}_n(x)),$$  
where 
\[\mathcal{T}_n: \RR^2\to  \RR^2, \quad \mathcal{T}_n(x)=x+n.\] 

When $f$ is not special, we can consider the \textit{$u$-accessible} class of a point as follows.  
\begin{defn}[\!\cite{GS22}\,]\label{2 def u-accessible class}
		The $u$-accessible class of $x$ for $f$ is defined by
		\begin{align*}
			{\rm Acc}^u(x):=\big\{ y\in \TT^2 \; | \;  \exists\, x=y^0, y^1,\,...\,,\, y^k=y \;\;
			&{\rm and} \;\; \bar{y}^i\in \TT^2_f \;\; {\rm with}\;\; (\bar{y}^i)_0=y^i \\
			&{\rm such \;\; that} \;\; y^{i+1}\in\mathcalf^u (y^i,\bar{y}^i), \forall 0\leq i\leq k-1 \big\}.
		\end{align*}
		We say $f$ is \textit{$u$-accessible}, if ${\rm Acc}^u(x)=\TT^2$ for all $x\in \TT^2$.
\end{defn}
The $u$-accessibility means that any two points on $\TT^2$ can be connected by the unstable manifolds.

\subsection{Rigidity of DA maps}

In this subsection, 
we present a dichotomy of DA maps and give a new rigidity result for $u$-accessibility.  

\begin{maintheorem}\label{thm d dichotomy}
	Let $f:\mtt$ be a $C^1$-smooth  DA map.  Then either $f$ is special, or $f$ is $u$-accessible.
\end{maintheorem}

By Theorem \ref{thm d dichotomy}, there is the following result.

\begin{main-cor}\label{2 cor both special}
	Let $f:\mtt$ be a $C^1$-smooth   DA map semi-conjugate to an Anosov map $g:\mtt$. 
	Then $f$ and $g$ are either both special or both $u$-accessible.
\end{main-cor}

We will prove Theorem \ref{thm d dichotomy} and Corollary \ref{2 cor both special} in Section \ref{sec: dichotomy}.
This dichotomy in Theorem \ref{thm d dichotomy} can be viewed as a continuation of the result in \cite{GS22} 
where the authors further assume that $f$ is Anosov in order to get the same dichotomy. 
Here, it should be mentioned that the proof in \cite{GS22} is invalid for our case, since we do not a priori know  the minimality of unstable manifolds of DA maps (see Remark \ref{3 rmk u is not dense}). 

Following Corollary \ref{2 cor both special},   
Theorem \ref{main-thm-conjugacy} is indeed decomposed into two cases: 
Theorem \ref{2 thm rigidity special} and Theorem \ref{2 thm rigidity accessible}.  
For an Anosov map $f$ with one-dimensional stable bundle, given a periodic point $p$ with period $n$, 
we define the stable Lyapunov exponent of $p$ by $$\lambda^s(p,f):=\frac{1}{n}{\rm log}\|Df^n|_{E^s_f(p)}\|.$$

\begin{thm}[\!\cite{AGGS23,GX2023}\,]\label{2 thm rigidity special}
	Let $f:\mtt$ be a $C^1$-smooth  DA map with linearization $A:\mtt$.  
	Then the following are equivalent:
	\begin{enumerate}
		\item $f$ is semi-conjugate to $A$ via $h:\mtt$;
		\item $f$ is special.
	\end{enumerate}
	Moreover, each of the two items implies that $h$ is a homeomorphism, $f$ is Anosov, 
	and $\lambda^s(p,f)\equiv\lambda^s(A)$ for all $p\in {\rm Per}(f)$.
	And, if $f$ is $C^r\ (r>1)$, 
	then the homeomorphism $h$ is $C^r$-smooth along stable manifolds.
\end{thm}

\begin{maintheorem}\label{2 thm rigidity accessible}
	Let $f:\mtt$ be a $C^1$-smooth  $u$-accessible  DA map.
	If $f$ is  semi-conjugate to a $C^1$-smooth $u$-accessible Anosov map $g:\mtt$ via $h:\mtt$. 
	Then $f$ is Anosov and $h$ is a homeomorphism. 
	Moreover, if  $f$ and $g$ are $C^r$-smooth $(r>1)$, 
	then the homeomorphism $h$ is $C^r$-smooth along stable manifolds.
\end{maintheorem}

We will prove Theorem \ref{2 thm rigidity accessible} in Section  \ref{sec rigidty}. We note in advance that the proof of the $C^1$-regularity case is different from \cite{GX2023,AGGS23,GS22}. 
Assuming Corollary \ref{2 cor both special} and Theorem \ref{2 thm rigidity accessible}, 
we can prove Theorem \ref{main-thm-conjugacy} as follows.

\begin{proof}[Proof of Theorem \ref{main-thm-conjugacy}]	
	Let $f,g:\mtt$ satisfy the assumption of  Theorem \ref{main-thm-conjugacy}. 
	Namely, $f$ is a $C^{1}$-smooth  DA map and semi-conjugate to a $C^{1}$-smooth Anosov map $g$ via $h:\mtt$. 
	By Corollary \ref{2 cor both special}, we consider the following two cases.
	\begin{itemize}
		\item Let $f$ and $g$ be both special. 
		By the semi-conjugacy $h$ homotopic to ${\rm Id}_{\TT^2}$ and $(h\circ f)_*=(g\circ h)_*$, we have $f_*=g_*$. 
		We assume that $f$ and $g$ admit the same linearization $A:\mtt$.   
		Since a special toral Anosov map is topologically conjugate to its linearization, 
		we consider $h_0:\mtt$ a conjugacy between $g$ and $A$. 
		Then, $f$ is semi-conjugate to $A$ via $h_1:=h_0\circ h$.
		Applying Theorem \ref{2 thm rigidity special}, we obtain that $f$ is Anosov, 
		and $h_1$ and $h$ are homeomorphisms. 
		In particular, when $f$ and $g$ are $C^r$-smooth $(r>1)$, 
		by Theorem  \ref{2 thm rigidity special},  
		both $h_1$ and $h_0$ are smooth along each stable manifold, so is $h$.
		
		\item Let $f$ and $g$ be both $u$-accessible. 
		The conclusion follows from Theorem \ref{2 thm rigidity accessible} directly.
	\end{itemize}
	This completes the proof of Theorem \ref{main-thm-conjugacy}.
\end{proof}

\subsection{Comparison to DA diffeomorphisms on $\mathbb{R}^3$}\label{subsec DA diffeo}

In this subsection, we clarify our result of Theorem \ref{main-thm-conjugacy} 
with the rigidity of DA diffeomorphism on $\TT^3$ in \cite{GS20,HS21}. 
Moreover, following this philosophy, 
we give a new semi-conjugacy rigidity result for DA diffeomorphism on $\TT^3$ 
which will be proved in Appendix \ref{sec-app-c}.

Let $f:\mtt$ be a DA map with linearization $A:\mtt$. 
Let $g$ be an Anosov map homotopic to $A$.   
Denote their partially hyperbolic splittings by 
\[
E^c_f\oplus E^u_f, 
\quad L^s_A\oplus L^u_A,
\quad E^s_g\oplus E^u_g. 
\]
Recall that $E^u_f$ and $E^u_g$ may depend on the choice of negative orbits.

Let $\phi:\TT^3\to \TT^3$ be a diffeomorphism homotopic to toral Anosov automorphism $B:\TT^3\to \TT^3$. 
Let $\psi$ be an Anosov diffeomorphism homotopic to $B$. 
Assume that they admit invariant partially hyperbolic splittings
\[
E^s_\phi\oplus E^c_\phi\oplus E^u_\phi,
\quad L^{s}_B\oplus L^{c}_B\oplus L^u_B,
\quad  E^{s}_\psi\oplus E^{c}_\psi\oplus E^u_\psi, 
\]
where $E_\psi^{s}$ and $L^{s}_B$ are strong stable bundles, 
while $E_\psi^{c}$ and $L^{c}_B$ are weak stable bundles. 
Recall that for the diffeomorphism case, 
there exist semi-conjugacies $h_B$ and $h_\psi$ from $\phi$ to $B$ and $\psi$, respectively.

\begin{thm}[\!\cite{HS21,GS20,HU14}\,]\label{thm DA diffeo}
	Let $\phi$ be given as above and $C^{1+\alpha}$-smooth. 
	Then there is the following dichotomy:
	\begin{itemize}
		\item either  $\phi$ is $su$-integrable, 
		i.e., the bundle $E^{s}_\phi\oplus E^u_\phi$ is joint-integrable.
		
		\item or $\phi$ is $su$-accessible, 
		i.e., any two points  can be connected by paths tangent to $E^{s}_\phi$ or $E^{u}_\phi$.
	\end{itemize}
Moreover, if $\phi$ is $su$-integrable, 
then $\phi$  is Anosov and  $h_B$ is a conjugacy smooth along  center manifolds.
\end{thm}

The pre-image sets of non-invertible maps $f, A, g$ can be regarded as their strong stable directions 
(see also \cite{AGGS23,GS22} for some discussion). 
Then $E^c_f$ and $ E^u_f$ correspond to $E^c_\phi$ and $E^u_\phi$, respectively. 
And the special property and $u$-accessibility of $f$ are 
corresponding to the $su$-integrability and $su$-accessibility of $\phi$, respectively. 
In particular, in the case of non-invertible maps on $\TT^2$, 
a conjugacy or a semi-conjugacy on $\TT^2$ automatically matches the pre-image sets. 
This can be seen as the conjugacy matches strong stable foliations in the case of diffeomorphisms on $\TT^3$.
Thus, Theorem \ref{main-thm-conjugacy} and Theorem \ref{thm d dichotomy} 
can be seen as a non-invertible version of Theorem \ref{thm DA diffeo}. 

It should mention again that all results contained in Theorem \ref{thm DA diffeo} need $C^{1+\alpha}$-regularity, 
while most results in our work are under the $C^1$ assumption. 
Since Theorem \ref{main-thm-conjugacy} includes the semi-conjugacy rigidity for the accessible case,  
we also adapt this semi-conjugacy rigidity to $\TT^3$ DA diffeomorphism.

\begin{thm}\label{thm appendix DA}
	Let $\phi$ and $\psi$ be given as above and  $C^{1+\alpha}$-smooth.
	If the semi-conjugacy $h_\psi$ maps the foliation $\mcf_\phi^s$ to $\mcf_\psi^s$. 
	Then, $\phi$ is an Anosov diffeomorphism and $h_\psi$ is a conjugacy smooth along  center manifolds.
\end{thm}

We will prove this additional result in Appendix \ref{sec-app-c}. 
Since there is the rigidity for $su$-integrable case in Theorem \ref{thm DA diffeo}, 
we mainly focus on the $su$-accessible case where the $C^1$ argument also works.

\subsection{Dynamics of DA on $\mathbb{R}^2$}\label{subsec DA}

We conclude the preliminary section by collecting some useful properties of DA maps.
Let $f:\mtt$ be a DA map with linearization $A:\mtt$.  
Let $F:\mrr$ be a lift of $f$.
We still denote a lift of $A:\mtt$ by $A:\mrr$,
and denote the stable/unstable bundles and foliations of $A$ 
by $L^{s/u}$ and $\mcl^{s/u}$ on $\TT^2$, 
$\widetilde{L}^{s/u}$ and $\tildeL^{s/u}$ on $\RR^2$, respectively. 

\begin{prop}[\!\cite{AGGS23}\,]\label{2 prop minimal foliation}
	The foliations $\mcl^{s/u}$ are minimal. 
	Moreover, for any $x\in\RR^2$ and $k\in \NN$, the set
	\[\bigcup_{n\in A^k\ZZ^2}\tildeL^{\sigma}(x+n)\]
	is dense in $\RR^2$, where $\sigma=s,u$.
\end{prop}

The following result is adapted from the well-known \cite{F70} 
in the case of toral DA diffeomorphisms.

\begin{prop}[\!\cite{AH94}, Theorem 8.2.1\,]\label{2 prop semi-conj in R2}
	There exists a unique surjection $H:\mrr$ satisfying:
	\begin{enumerate}
		\item $H\circ F=A\circ H$;
		\item There exists $C>0$ such that $\| H-{\rm Id}_{\RR^2} \|_{C^0}<C$.
	\end{enumerate}
	Moreover, $H$ is uniformly continuous. 
	Assume further that $f$ is Anosov, then $H$ is uniformly bi-continuous.
\end{prop}

We say that $H:\RR^2\to \RR^2$ is \emph{commutative with the deck transformations} (or \emph{$\ZZ^2$-periodic})
if \[H\circ \mathcal{T}_n=\mathcal{T}_n\circ H,  \quad\forall\ n\in\ZZ^2,\]
where $\mathcal{T}_n:\RR^2\to \RR^2$ is a deck transformation for $n\in\ZZ^2$.

\begin{rmk}\label{2 rmk H to h}
	A semi-conjugacy $H$ in Proposition \ref{2 prop semi-conj in R2} is commutative with the deck transformations if and only if   it can descend to $\TT^2$ and hence induce a semi-conjugacy between $f$ and $A$.
\end{rmk}

Modulo the stable foliation, 
$H$ keeps commutativity with deck transformations. 
Note that $H$ is ``approaching'' commutative with $A^k\ZZ^2$ as $k\to +\infty$. 
Thus, we have the following property.

\begin{prop}[\!\cite{AGGS23,GX2023}\,]\label{2 prop semiconj on stable}
	Let $H$ be given in Proposition \ref{2 prop semi-conj in R2}. Then 
	\begin{enumerate}
		\item For any $x\in\RR^2$ and $n\in\ZZ^2$, $H(x+n)-n\in \tildeL^s\big(H(x)\big)$.
		\item There exists $C>0$ such that for any $x\in\RR^2$, $k\in\NN$ and $n\in A^k\ZZ^2$, 
		\[\Big| H(x+n)-H(x)-n \Big|<2C\cdot \|A|_{L^s}\|^k.\]
	\end{enumerate}
\end{prop}

The following proposition collects several useful properties of the invariant foliations of $F$, 
which originate from \cite{AGGS23,HH21,HS21,U12},  and one can find the proof of present adaption in  \cite{GX2023}.

\begin{prop}[\cite{GX2023}]\label{2 prop foliation on R2}
	Let $f:\mtt$ be a  DA map. Let $F$ be the lift of $f$ on $\RR^2$.   Then
	\begin{enumerate}
		\item The center bundles $E^c_f$ and $E^c_F$ are integrable to foliations $\mcf_f^c$ and $\mcf_F^c$ respectively;
		
		\item The pair of foliations $\mcf_F^c$ and $\mcf_F^u$ admits the Global Product Structure, 
		i.e., for any $x,y\in\RR^2$, the manifold $\mcf_F^c(x)$ intersects with $\mcf_F^u(y)$ exactly once;
		
		\item The foliation $\mcf_F^\sigma \ (\sigma=c,u)$ is quasi-isometric, 
		i.e., there exists $C_1,C_2>0$ such that 
		\[d_{\mcf_F^\sigma}(x,y)<C_1|x-y|+C_2, \quad \forall x\in\RR^2\  {\rm and}\ \forall y\in \mcf_F^\sigma(x). \] 
	\end{enumerate}
	Moreover, assume further that the DA map $f$ is semi-conjugate an Anosov map $g$ whose lift on $\RR^2$ is $G$. 
	Then there is a unique semi-conjugacy $H:\mrr$ of $F$ and $G$ such that
	\begin{enumerate}
		\item There exists $C>0$ such that $\| H-{\rm Id}_{\RR^2} \|_{C^0}<C$.
		
		\item  $H$ preserves the unstable foliation, 
		i.e., $H(\mcf_F^u(x))=\mcf_G^u(H(x))$, for all $x\in\RR^2$. 
		Furthermore, the restrictions of $H$ to unstable manifolds are uniformly bi-H\"older continuous homeomorphisms. 
		Here, the uniformity means that the H\"older constants are independent of the manifolds and points.
		
		\item  $H$ preserves the center foliation. Furthermore, 
		\begin{itemize}
			\item For points $x, y\in\RR^2$, $y\in \mcf^c_F(x)$ if and only if $H(y)\in \mcf_G^s(H(x))$.
			
			\item The restriction $H|_{\mcf^c_F(x)}:\mcf^c_F(x)\to \mcf_G^s(H(x))$ is monotonic. 
			In particular, for any point $x\in\RR^2$, the set $H^{-1}(x)$ is
			either a single point or a uniformly compact local center leaf, 
			where the uniformity means that there is a uniform upper bound for the length of $H^{-1}(x)$ for every $x\in\RR^2$.
		\end{itemize}
		
		\item Let $\Gamma=\big\{x\in\RR^2\ |\ H^{-1}\circ H(x)=\{x\} \big\}$ be the set of $H$-injection points,
		and $\bar{\Gamma}$ be the closure of $\Gamma$. Then
		\begin{itemize}
			\item  The set $\bar{\Gamma}$ is unstable manifold saturated, 
			i.e.,  $\mcf^u_F(x)\subset \bar{\Gamma}$, for all $x\in\bar{\Gamma}$;
			
			\item  The restriction  $H|_{\mcf^c_F(x)}$ is not locally constant at all $x\in\bar{\Gamma}$.
		
			\item The set $\bar{\Gamma}$ is $\ZZ^2$-periodic, 
			i.e., $\bar{\Gamma}+\ZZ^2=\bar{\Gamma}$.
		\end{itemize}
	\end{enumerate}
\end{prop}

\begin{proof}	
	When $g$ is the linearization $A$ of $f$, the whole proposition has been proved in \cite{GX2023}, 
	except for the uniformly H\"older continuity of $H|_{\mcf_{F}^u}$ 
	whose proof is classical and can be seen in \cite[Theorem 19.1.2]{KH95} for example.  
	Even if one considers that $f$ is just homotopic to $g=A$,  
	the proposition also holds except for the $\ZZ^2$-periodicity of $\bar{\Gamma}$ (see \cite{GX2023}).   
	
	The properties on foliations $\mcf_{F}^{c/u}$ are independent of $g$, and hence holds as proved in \cite{GX2023}. 
	The properties of semi-conjugacy between $F$ and $G$, in the present case, 
	inherit from the ones of semi-conjugacy between $F$ and $A$. 
	Here, we give some explanation on this.
	
	Since $f$ is semi-conjugate to $g$ via $h$ homotopic to ${\rm Id}_{\TT^2}$, we have $f_*=g_*$. 
	Then we can assume that  $f$ and $g$ admit the same linearization $A$.  
	Let $H$ be a lift of $h$ such that it is a semi-conjugacy of $F$ and $G$. 
	Then $H$ is bounded from the identity.  
	One can see that the lift $H$ is the unique semi-conjugacy that we need. 
	And $H=H_G^{-1}\circ H_F$, where $H_F, H_G:\mrr$ are given by Proposition \ref{2 prop semi-conj in R2} 
	with $H_{\sigma}\circ \sigma=A\circ H_\sigma$ for $\sigma=F,G$.

	It is known that $H_G$ is a bi-H\"older homeomorphism which preserves the stable and unstable foliations. 
	Then, $H=H_G^{-1}\circ H_F$ and  $\bar{\Gamma}$ inherit the properties of the case in \cite{GX2023}
	where $f$ is homotopic to $g=A$, except for the  $\ZZ^2$-periodicity of $\bar{\Gamma}$.
	To show  $\bar{\Gamma}+\ZZ^2=\bar{\Gamma}$, it suffices to prove that $\Gamma+\ZZ^2=\Gamma$. 
	Let $x\in \Gamma, n\in \ZZ^2$, and $z\in H^{-1}\circ H(x+n)$.  
	Since $H$ is $\ZZ^2$-periodic, $H(x)=H(x+n)-n=H(z)-n=H(z-n)$.
	By assumption that $x\in \Gamma$ is an $H$-injection point, one has that $x=z-n$. 
	Thus, $H^{-1}\circ H(x+n)=\{z\}$ and $(x+n)\in \Gamma$.
\end{proof}

For a one-dimensional foliation $\mcf$ on $\RR^2$, define the holonomy map
\[{\rm Hol}^{\mcf}_{x,y}: \mcl(x)\to \mcl(y),  \]
by ${\rm Hol}^{\mcf}_{x,y}(z)=\mcf(z)\cap \mcl(y)$, for $z\in\mcl(x)$,
where $\mcl(x)$ and $\mcl(y)$ are transversals of $\mcf$ and respectively intersect with each leaf of $\mcf$ at most once.
Since the foliations $\mcf^u_F$ and $\mcf^c_F$ admit the Global Product Structure for the DA map $f$,  
the holonomy map of $\mcf^{u}_F$ given by 
\[{\rm Hol}^{\mcf^u_F}_{x,y}: \mcf^c_F(x)\to \mcf^c_F(y),  \]
is well-defined.  
For a size $R>0$, denote 
\[ \mcf^u_F(x,R):=\{y\in\mcf^u_F(x)\ |\ d_{\mcf^u_F}(x,y)<R \},\]
where $d_{\mcf^u_F}(\cdot,\cdot)$ is induced by the metric restricted to the unstable manifold. 
It is well-known that ${\rm Hol}^{\mcf^u_F}$ is H\"older continuous, 
and is $C^1$ when $f$ is $C^{1+\alpha}$-smooth (see \cite[Theorem 7.1]{Pesinbook04} for example).

\begin{prop}[\!\cite{Pesinbook04}\,]\label{2 prop Holder foliation}
	Let $f$ be a $C^1$-smooth DA map on $\TT^2$. 
	Then the holonomy maps of $\mcf^u_F$ are uniformly H\"older continuous, 
	namely, for any size $R>0$ there are constants $0<\beta<1$ and $C>1$ such that
	the holonomy maps induced by local foliation $\mcf^{u}_F(\cdot, R)$ are $(\beta,C)$-H\"older continuous.
	
	Moreover, if $f$ is $C^{1+\alpha}$-smooth, 
	then the holonomy maps of $\mcf^u_F$ are uniformly  $C^1$-smooth. 
	In particular, the upper bound of the derivative $\|D{\rm Hol}^{\mcf^u_F}_{x,y} \|$ depends only on 
	the size $R$ of foliation $\mcf^u_F$.
\end{prop}

\section{$u$-Accessibility of DA maps}\label{sec: dichotomy}

In this section, we prove Theorem \ref{thm d dichotomy}, 
i.e., a DA map on $\TT^2$ is either special, or $u$-accessible. 
Moreover, we give a quantitative estimation for the $u$-accessibility: 
any two points on $\TT^2$ can be connected by finitely many manifolds $\mcf_f^u$ 
which are projected from $\RR^2$ and have uniform size. 
This quantitative $u$-accessibility is also useful 
for the proof of Theorem \ref{2 thm rigidity accessible} in the next section.
   
\begin{thm}\label{3  thm dichotomy}
	Let $f:\mtt$ be a $C^1$-smooth DA map. 
	Then there is the following dichotomy:
	\begin{itemize}
		\item either $f$ is special.
		
		\item or $f$ is $u$-accessible. 
		Moreover, there are constants $R>0,N\in\NN$ and a point $w_0\in \TT^2$ with a lift $w^*\in\RR^2$	
		such that for any $x\in\TT^2$, there exist 
		\begin{enumerate}
			\item[-] $k$ points $x_i\in\TT^2 \ (0\leq i\leq k\leq N)$ with $x_k=w_0$ and $x_0=x$,
			\item[-] lifts $x'_i\in\RR^2$  of $x_i$ with $x'_k=w^*$,
		\end{enumerate}
		satisfying $x_{i+1}\in \pi\big( \mcf^u_F(x'_i,R)\big)$ for any $0\leq i\leq k-1$. 
   \end{itemize}  
\end{thm}

Assuming Theorem \ref{3  thm dichotomy}, we first give the proof of Corollary \ref{2 cor both special}, 
i.e., if a DA map $f$ on $\TT^2$ is semi-conjugate to an Anosov map $g$, 
then they are simultaneously special or $u$-accessible.

\begin{proof}[Proof of Corollary \ref{2 cor both special}]
	Let $F,G:\mrr$ be two lifts of $f$ and $g$ respectively. 
	Let $H:\mrr$ be a lift of $h$ such that $H\circ F=G\circ H$. 
	Then $H(x+n)=H(x)+n$, for all $x\in\RR^2$ and $n\in\ZZ^2$.

	If $f$ is special, then 
	\[
	\mcf_F^u(x+n)=\mcf_F^u(x)+n,\quad \forall\, x\in\RR^2, n\in\ZZ^2.
	\]
	For any point $y\in \RR^2$, there exists a point $x\in H^{-1}(y)$. 
	Since $H$ is commutative with the deck transformations and
	preserves the unstable foliation (Proposition \ref{2 prop foliation on R2}),   
	one has for every $n\in\ZZ^2$,
	\begin{align*}
		\mcf_G^u(y+n)=  \mcf_G^u(H(x)+n) 
		&= \mcf_G^u(H(x+n))\\
		&= H\big(  \mcf_F^u(x)+n  \big)= H\big(  \mcf_F^u(x)  \big)+n =\mcf_G^u(y)+n.
	\end{align*}
	Hence $g$ is also special.
	
	If $f$ is non-special, then by Theorem \ref{3  thm dichotomy} $f$ is $u$-accessible. 
	It remains to prove that $g$ is also $u$-accessible. 
	Let $x,y\in\TT^2$. Let $x'\in h^{-1}(x)$ and $y'\in h^{-1}(y)$. 
	By Theorem \ref{3  thm dichotomy} again, 
	we assume that $\mcf^i \subset \RR^2 \ (1\leq i\leq k)$ are leaves of $\mcf^u_F$  
	such that the paths $\pi(\mcf^i)$ connect $x'$ and $y'$. 
	Since $H$ preserves the unstable manifolds, 
	$H(\mcf^i)$ are leaves of $\mcf^u_G$ and the paths $\pi\big(H(\mcf^i)\big)$ connect $x$ and $y$. 
	Hence $g$ is also $u$-accessible.
	
	This completes the proof of Corollary \ref{2 cor both special}.
  \end{proof}
  
Now we prove Theorem \ref{3  thm dichotomy} and hence get Theorem \ref{thm d dichotomy}.  
We divide the proof into two parts for clarity.

\subsection{Proof of the dichotomy part of Theorem \ref{3  thm dichotomy}}\label{subsec 3.1}

It is clear that if $f$ is special, then $f$ is not $u$-accessible.  
Conversely, let $f$ be not $u$-accessible. 
We will prove that $f$ is semi-conjugate to its linearization $A:\mtt$, 
then conclude that $f$ is special by Theorem \ref{2 thm rigidity special}. 
Moreover, by Remark \ref{2 rmk H to h}, 
it suffices to show that the semi-conjugacy $H:\RR^2\to\RR^2$ between the lifts $F$ and $A$ is commutative with the deck transformations. 
	
\begin{lem}\label{3 lem x_0}
	Let $f$ be not $u$-accessible. 
	Then there exists $x_0\in\RR^2$ such that 
	\begin{align}
		\mcf_F^u(x_0+n)=\mcf_F^u(x_0)+n, \quad \forall\, n\in\ZZ^2. \label{eq. x_0}
	\end{align}
\end{lem}
 
\begin{proof}[Proof of Lemma \ref{3 lem x_0}]
	Fix  $x_0\in \RR^2$, 
	if there exists $n_0\in\ZZ^2$ such that  
	\begin{align}
		\mcf_F^u(x_0+n_0)-n_0\neq \mcf_F^u(x_0), \label{eq. 3. xn_0}
	\end{align}
	then we claim that $\pi(x_0)\in\TT^2$ is an interior point of its accessible class,
	i.e., $\pi(x_0)\in {\rm Int}\big({\rm Acc}^u(\pi(x_0))\big)$. 
	Thus, by the connectivity of $\TT^2$, every point is an interior point of its accessible class 
	and hence $f$ is $u$-accessible, 
	which contradicts the assumption.

	First, let $x_0\in\RR^2$ and $n_0\in\ZZ^2$ satisfy \eqref{eq. 3. xn_0}. 
 
	\begin{claim}\label{3 claim y_0}
		There exists $y_0\in \mcf^u_F(x_0)$ such that $D\mathcal{T}_{n_0}E^u_F(y_0)\neq E^u_F(y_0+n_0)$, where $\mathcal{T}_{n_0}(x)=x+n_0$.
	\end{claim}
	\begin{proof}[Proof of Claim \ref{3 claim y_0}]
		Otherwise, if each point $y\in  \mcf^u_F(x_0)$ has $D\mathcal{T}_{n_0}E^u_F(y)= E^u_F(y+n_0)$, 
		then 
		\[ T_{y+n_0}\big(\mcf^u_F(x_0)+n_0 \big)=D\mathcal{T}_{n_0}T_y\big(\mcf^u_F(x_0) \big) =D\mathcal{T}_{n_0}E^u_F(y)= E^u_F(y+n_0), \quad \forall\, y\in \mcf^u_F(x_0). \]
		By the unique integrability of the unstable bundle $E^u_F$ (see \cite{Pesinbook04} for example), 
		we get $\mcf^u(x_0)+n_0=\mcf^u_F(x_0+n_0)$ which contradicts \eqref{eq. 3. xn_0}.
	\end{proof}

	Then, let $y_0$ be given by Claim \ref{3 claim y_0}. 
	\begin{claim}\label{3 claim int point}
		The point $\pi(y_0)$ is an interior point of Acc$^u(\pi(y_0))$. 
		So is the point $\pi(x_0)$.
	\end{claim}

	\begin{proof}[Proof of Claim \ref{3 claim int point}]
		For short, let $x=\pi(x_0)$ and $y=\pi(y_0)$.
		Let $n_0$ be given as above. 
		Denote by $\bar{y}^0, \bar{y}^{n_0}\in \TT^2_f$ 
		the $f$-orbits of the projections of the $F$-orbits of $y_0$ and $y_0+n_0$ respectively. 
		Then, Claim \ref{3 claim y_0} implies that $$E^u_f(y,\bar{y}^0)+ E^u_f(y,\bar{y}^{n_0})=T_y\TT^2.$$ 
		By continuity,    
		there exists a neighborhood $U_0\subset \RR^2$ of $y_0$  
		such that the projections of the local unstable foliations in $U_0$ and $U_0+n_0$ are transverse 
		with each other in $\pi(U_0)$. 
		Hence $y\in\pi(U_0) \subset {\rm Int}\big( {\rm Acc}^u(y)\big)$. 
		
		Moreover, since $y_0\in\mcf_F^u(x_0)$, $x\in \mcf^u_f(y,\bar{y}^0)$. 
		Then, there exists a neighborhood $U_x\subset \TT^2$ of $x$ such that 
		\[U_x\subset \bigcup_{z\in U_0}\pi \big(  \mcf^u_F(z)\big). \]
		 In particular, $x\in U_x\subset {\rm Acc}^u\big(\pi(U_0)\big)={\rm Acc}^u(y)={\rm Acc}^u(x)$, 
		 and $x$ is an interior point of ${\rm Acc}^u(x)$.	
	\end{proof}
	
	This completes the proof of Lemma \ref{3 lem x_0}. 
\end{proof}

\begin{rmk}\label{3 rmk u is not dense}
	Let $x_0$ be given in Lemma \ref{3 lem x_0}. 
	Then every point $x\in \pi\big(\mcf^u_F(x_0)\big)$ is a special point.
	Note that $\pi\big(\mcf^u_F(x_0)\big)$ may be not dense in $\TT^2$, 
	though it is dense when $f$ is Anosov \cite{GS22}. 
	If it is dense, then one can immediately get that $f$ is special, 
	by the continuity of unstable bundle with respect to the orbits.  
	This is the main reason why the proof of the same dichotomy in \cite{GS22} fails here.
\end{rmk}

We continue to prove that $H$ is commutative with the deck transformations.  
Without loss of generality, 
we assume that $F$ preserves the orientation of each center leaf. 

Let $\Gamma=\big\{x\in\RR^2\ |\ H^{-1}\circ H(x)=\{x\} \big\}$ be the $H$-injective set. 
Note that $\Gamma$ is an $F$-invariant set, so is its closure $\bar{\Gamma}$. 
Moreover, ${\rm Fix}(F)\cap \bar{\Gamma}\neq \emptyset$. 
Indeed, let $0\in\RR^2$ be the unique fixed point of $A$. 
By Proposition \ref{2 prop foliation on R2}, $H^{-1}(0)$ is 
either a single point and hence a fixed point of $F$, 
or a compact local center leaf 
whose endpoints are fixed points of the orientation-preserving diffeomorphism $F$ restricted to this curve.

\begin{lem}\label{3 lemma xn}
	For any $x\in \bar{\Gamma}$ and $n\in\ZZ^2$,  
		\[\mcf_F^u(x+n)=\mcf_F^u(x)+n,     \quad \forall\, n\in\ZZ^2.\]
		In particular, there exists a fixed point $z_0$ of $F$ such that 
		\[\mcf_F^u(z_0+n)=\mcf_F^u(z_0)+n, \quad \forall\, n\in\ZZ^2.\]
 \end{lem}
\begin{proof}[Proof of Lemma \ref{3 lemma xn}]
	   Let $x_0$ be given by Lemma \ref{3 lem x_0} and consider the set
	   \begin{align}
		\mathbf{F}_{x_0}:=\bigcup_{n\in\ZZ^2} \mcf^u_F(x_0+n).\label{eq. mathbf F}
	   \end{align}
	\begin{claim}\label{3 claim closure}
		$\bar{\Gamma} \subset \bar{\mathbf{F}}_{x_0}$, the closure of $\mathbf{F}_{x_0}$.
	\end{claim}
	
	\begin{proof}[Proof of Claim \ref{3 claim closure}]
		Let $x\in \bar{\Gamma}$. 
		Since $H|_{\mcf_F^c(x)}$ is not locally constant at $x$, 
		consider a short open segment $J^c\subset \mcf^c_F(x)$ containing $x$, 
		we may choose a non-trivial closed interval $I^c\subset J^c$ with $H^{-1}\circ H(I^c)=I^c$
		such that $H(I^c)$ has positive length.
		By Proposition \ref{2 prop minimal foliation} and the second item of Proposition \ref{2 prop semiconj on stable}, 
		the set
		\begin{align}
	    	H(\mathbf{F}_{x_0})=
			H(\bigcup_{n\in\ZZ^2}\mcf_F^u(x_0+n))=
			\bigcup_{n\in\ZZ^2}\tildeL^u\big(H(x_0+n)\big) \label{3 eq. closure}
		\end{align}
		is dense in $\RR^2$. 
		Hence there is $n_0\in\ZZ^2$ such that the leaf $\tildeL^u\big(H(x_0+n_0)\big)={H}\big(\mcf_F^u(x_0+n_0)\big)$
		intersects $H(I^c)$ at a point $z$. 
		Thus, $H^{-1}(z)\subset I^c\subset  J^c$, 
		and $\mcf_F^u(x_0+n_0)$ intersects $J^c$. 
		This implies $\bar{\Gamma} \subset \bar{\mathbf{F}}_{x_0}$.
	\end{proof}
  
	Let $x\in \bar{\Gamma}$. By Claim \ref{3 claim closure}, 
	there exist $n_k\in\ZZ^2$ and $y_k\in \mcf^u_F(x_0+n_k)$ such that $y_k\to x$ as $k\to +\infty$. 
	On the one hand, for any $n\in\ZZ^2$, 
	\[\mcf_F^u(y_k+n)\to \mcf_F^u(x+n), \quad {\rm as}\ k\to +\infty. \]
	On the other hand, by Lemma \ref{3 lem x_0},
	\[\mcf_F^u(y_k+n)=\mcf_F^u(x_0+n_k+n)=\mcf_F^u(x_0+n_k)+n=\mcf_F^u(y_k)+n\to \mcf_F^u(x)+n, \quad {\rm as}\ k\to +\infty. \]
	We get that $\mcf_F^u(x+n)=\mcf_F^u(x)+n$.  
	In particular, let $z_0\in{\rm Fix}(F)\cap \bar{\Gamma}\neq \emptyset$, 
	then $\mcf_F^u(z_0+n)=\mcf_F^u(z_0)+n$ for all $n\in\ZZ^2$. 
  
	This ends the proof of Lemma \ref{3 lemma xn}.
\end{proof}

Then the existence of the semi-conjugacy between $f$ and $A$ follows from the following proposition.

\begin{prop}\label{2 prop special implies semi-conjugacy}
	Let $z_0$ be a fixed point of $F$ such that $\mcf_F^u(z_0+n)=\mcf_F^u(z_0)+n$ for all $n\in\ZZ^2$. 
	Then the semi-conjugacy $H$ is commutative with the deck transformations.
\end{prop}

Proposition \ref{2 prop special implies semi-conjugacy} is similar to the result of \cite{GX2023}  
in which the authors assume that $f$ is already special and then show that $f$ is semi-conjugate to $A$.   
Here, we just need one special unstable manifold to make sure the existence of the semi-conjugacy.  
The proof of Proposition \ref{2 prop special implies semi-conjugacy} is similar to the argument in \cite{GX2023}, 
which we leave to Appendix \ref{sec-app-a}.

By Proposition \ref{2 prop special implies semi-conjugacy},  
we complete the proof of the dichotomy part of Theorem \ref{3  thm dichotomy}.

\subsection{Proof of the estimation part of Theorem \ref{3  thm dichotomy}}\label{subsec 3.2}

In the following, we assume that $f$ is $u$-accessible 
and give the estimations of the size and the number of unstable manifolds connecting two points.  

Fix $n\in\NN$ and $r>0$. 
For any $x\in\TT^2$, we denote by Acc$^u(x,n,r)$ the set consisting of points 
which can reach $x$ within $n$ unstable manifolds (with respect to orbits) 
with size (on the leaf of $\mcf^u_F$) no more than $r$. 

\begin{claim}\label{3 claim N0R0}
	There exist $y\in\TT^2$, $N_0\in\NN$, and $R_0>0$ such that  $\TT^2= {\rm Acc}^u(y,N_0,R_0)$.
\end{claim}

\begin{proof}[Proof of Claim \ref{3 claim N0R0}]
	The idea of this proof originates from that for the diffeomorphism case \cite{D2003}.
	Let $y=\pi(y_0)$ and $U=\pi(U_0)$ be given by Claim \ref{3 claim y_0} and Claim \ref{3 claim int point}. 
	It follows from the proof of Claim \ref{3 claim int point} that there exists $R_0>0$ such that 
	\begin{align}
		U\subset {\rm Acc}^u(y,2,R_0). \label{eq. 3. accset 0}
	\end{align}
	Indeed, the number $2$ is given by 
	the two transverse local unstable foliations (projected from $\RR^2$) in $U$, 
	and the size $R_0>0$ is given by the size of $U$. 
	Then, by the assumption that $f$ is $u$-accessible,  
	\begin{align}
		\TT^2=\bigcup_{n\geq 1}  {\rm Acc}^u(y,n,R_0)=\bigcup_{n\geq 1} \bigcup_{x\in U}  {\rm Acc}^u(x,n,R_0). \label{eq. 3. accset 1}
	\end{align}
	Note that we can fix $R_0$ to get the accessibility, 
	since a unstable manifold with large size can be decomposed into finitely many unstable manifolds within size $R_0$.
	
	It is clear that $\bigcup_{x\in U}{\rm Acc}^u(x,n,R_0)$ is an open set. 
	Indeed, if $z\in \bigcup_{x\in U}{\rm Acc}^u(x,n,R_0)$, 
	then there exist $x\in U$ and $k\ (1\leq k\leq n)$ unstable manifolds within size $R_0$ 
	which we denote by $\mcf^{i}\ (1\leq i\leq k)$,
	such that $\mcf^{i}\cap \mcf^{i+1}\neq \emptyset$ with $x\in \mcf^1$ and $z\in \mcf^k$. 
	By the continuity of unstable manifolds with respect to the orbits, 
	there exists a neighborhood $U_x\subset U$ of $x$ such that 
	each point in $U_x$ can be sent to a neighborhood of $U_z$ of $z$ 
	along unstable manifolds which are $C^1$-close to $\mcf^i\ (1\leq i\leq k)$. 
	It follows that $U_z\subset \bigcup_{x\in U}{\rm Acc}^u(x,n,R_0)$.
	
	By the compactness of $\TT^2$ and  \eqref{eq. 3. accset 0}, 
	the formula \eqref{eq. 3. accset 1} implies that there is $N_1\in\NN$ such that 
	\begin{align}
		  \TT^2= \bigcup_{x\in U}{\rm Acc}^u(x,N_1,R_0) \subset {\rm Acc}^u(y,N_1+2,R_0). \label{eq. 3. accset 2}
	\end{align}
	Hence, $N_0=N_1+2$ and $R_0$ are the constants satisfying Claim \ref{3 claim N0R0}.
\end{proof}

Claim \ref{3 claim N0R0} indeed gives a quantitative estimation for the accessibility 
with respect to orbits 
in the inverse limit space $\TT^2_f$.  
Finally, it remains to show that the estimation also holds on the universal cover $\RR^2$. 

Let $y$, $U$, $N_0$,  $N_1$ and $R_0$ be given by Claim \ref{3 claim N0R0}.

 \begin{figure}[htbp]
	\centering
	\includegraphics[width=15.6cm]{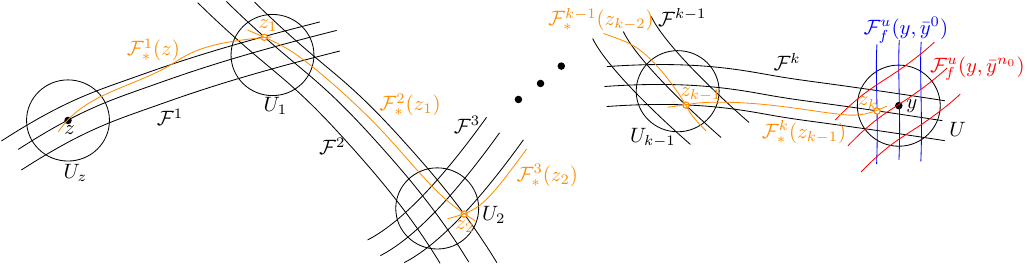}
	\caption{Change the unstable manifolds into projections}	
	\label{figure-uacc}
\end{figure}

Note that the proof of Claim \ref{3 claim N0R0} with \eqref{eq. 3. accset 2} implies that 
for any $z\in\TT^2$ there exists a small neighborhood  $U_z$ of $z$ 
and $k$-families $(1\leq k\leq N_1)$ of local unstable foliations  $\mcf^i\ (1\leq i\leq k)$ 
such that $\mcf^i\cap \mcf^{i+1}\neq \emptyset$ and they send $U_z$ into $U$.  
Let $z\in \mcf^1(z)$ and denote by $U_1$ the image of $U_z$ under the holonomy of $\mcf^1$.  
Similarly, we denote by $U_i \ (2\leq i \leq k)$ the image of $U_{i-1}$ under the holonomy of $\mcf^{i}$.  
By the denseness (in $\TT^2_f$) of the projection of $F$-orbit on $\RR^2$ (see Remark \ref{2 rmk universal to limit}) 
and the continuity of unstable manifolds, 
there exists an unstable manifold $\mcf^1_*(z)$ satisfying
\begin{enumerate}
	\item[-] $\mcf^1_*(z)$ is a projection of  an unstable manifold in $\RR^2$;
	\item[-] $\mcf^1_*(z)$ is $C^1$-close to $\mcf^1(z)$  and  within size $2R_0$;
	\item[-] There exists a point $z_1\in U_1\cap \mcf^1_*(z)$.
\end{enumerate}
See Figure \ref{figure-uacc}.
By induction, we apply the above method to points $z_i \ (1\leq i\leq k-1)$ and manifolds $\mcf^i(z_i)$ 
to get the manifold $\mcf^{i+1}_*(z_i)$ and the point $z_{i+1}\in U_{i+1}\cap \mcf^{i+1}_*(z_i)$. 
In particular, each unstable manifold $\mcf^{i}_*$ is projected from $\RR^2$ and its size is no more than $2R_0$. 
By the choice of  $y$ and  $U$ (Claim \ref{3 claim int point}),  
the point $z_k\in U$ can reach to $y$ by two unstable manifolds projected from $\RR^2$ within size $R_0$.  

Hence, by taking $w_0=y=\pi(y_0)$,  $w^*=y_0$, $N=N_0=N_1+2$, and $R=2R_0$, 
we complete the proof of the estimation part of Theorem \ref{3  thm dichotomy}.

\section{Rigidity for $u$-accessible DA maps}\label{sec rigidty}

In this section, we prove Theorem \ref{2 thm rigidity accessible}.
Assume that $f$ is a DA map, $g$ is an Anosov map,
$f$ and $g$ are both $u$-accessible and semi-conjugate via $h:\mtt$. 
Recall that $F,G$ and $H$ are lifts of $f,g$ and $h$ on $\RR^2$. 

For clarity, we divide the proof into two parts: 
In Subsection \ref{subsec C1 case}, 
we present the proof of the $C^1$ part of Theorem \ref{2 thm rigidity accessible}, 
i.e., $f$ is Anosov and $h$ is a conjugacy; 
In Subsection \ref{subsec C1+ case}, 
we complete the proof of the $C^{1+\alpha}$ part of Theorem  \ref{2 thm rigidity accessible}, 
i.e., when $f$ and $g$ are $C^{r}$-smooth $(r>1)$, 
the conjugacy $h$ is $C^r$-smooth along the stable manifolds.  Although, by $C^1$-case, we already got an Anosov map $f$ and then the $C^{1+\alpha}$-case follows from the main result of \cite{GS22} (see Remark \ref{1 rmk compare to endo}).
Here, we provide a  proof for the $C^{1+\alpha}$-case which depends on neither the $C^1$-case nor the work \cite{GS22}.

\subsection{Proof of the $C^1$-regularity part of Theorem \ref{2 thm rigidity accessible}}\label{subsec C1 case}

	\begin{lem}\label{4 lem conjugacy}
		The semi-conjugacy $h:\mtt$ is indeed a conjugacy.
	\end{lem}
	
	\begin{proof}[Proof]
		By contradiction, we assume that $H$ is not injective. 
		Let $\Gamma$ be the set of $H$-injection points. 
		By Proposition \ref{2 prop foliation on R2}, 
		there are points $x\in \Gamma$ and $y\in \RR^2\setminus \bar{\Gamma}$. 
		It follows from Theorem \ref{3  thm dichotomy} that 
		there are points $x=x_0, x_1, x_2, \cdots, x_k\in \RR^2$,
		and $n^*, n_0, n_1,n_2\cdots n_k\in \ZZ^2$ and $R>0$ 
		such that 
		\[ 
		y+n^*\in \mcf^u_F(x_k+n_k,R) 
		\quad {\rm and}\quad x_{i+1}\in \mcf^u_F(x_i+n_i,R),
		\quad \forall\, 0\leq i\leq k-1.
		\]
		Since the set $\bar{\Gamma}$ is $\ZZ^2$-periodic and unstable manifold saturated 
		(see also Proposition \ref{2 prop foliation on R2}), 
		$x=x_0\in \Gamma$ implies that $x_1\in \bar{\Gamma}$.  
		Inductively, $y+n^*\in \bar{\Gamma}$, 
		and hence $y\in \bar{\Gamma}$. 
		This contradicts the choice $y\in\RR^2\setminus\bar{\Gamma}$.
	\end{proof}

	\begin{lem}\label{4 lem Holder conjugacy}
		The conjugacy $h:\mtt$ is bi-H\"older continuous.
	\end{lem}
	
	\begin{proof}[Proof of Lemma \ref{4 lem Holder conjugacy}]
	
	By Proposition \ref{2 prop foliation on R2},
	the restriction of $H$ to unstable manifolds are uniformly bi-H\"older continuous, 
	and $\mcf_F^u$ and $\mcf_F^c$ admit the Global Product Structure. 
	To show that $H$ is bi-H\"older continuous,	
	it suffices to show that the restrictions of $H$ to center manifolds are also uniformly bi-H\"older continuous. 
	In the following, we only prove that $H|_{\mcf_{F}^c}$ is H\"older continuous. 
	The H\"older continuity of $H^{-1}|_{\mcf^s_G}$ can be obtained by the same method.
	
	By Lemma \ref{4 lem conjugacy} and Proposition \ref{2 prop foliation on R2}, 
	the map $H$ restricted to a center manifold $\mcf_F^c(x_0)$ is a monotonic homeomorphism. 
	Thus, there exists a Lipschitz point of $H|_{\mcf_F^c}$. 
	Without loss of generality, we assume that $x_0$ is a Lipschitz point. 
	
	We say that the map $H|_{\mcf_F^c}$ is H\"older continuous at point $x$, 
	if there are H\"older constants $0<\alpha<1, C>1$, 
	and a H\"older neighborhood $U^c(x)\subset \mcf_F^c(x)$ 
	such that 
	\[
	d_{\mcf_G^s}\big(H|_{\mcf_F^c}(x),H|_{\mcf_F^c}(y)\big)\leq C\cdot d_{\mcf_F^c}^\alpha(x,y), 
	\quad \forall\, y\in U^c(x), 
	\]
	where $d_{\mcf_{F/G}^{c/s}}(\cdot,\cdot)$ is induced by the metric restricted to the foliation $\mcf_{F/G}^{c/s}$.
	
	\begin{claim}\label{4 claim Holder Z2}
		Let  $x$ be a H\"older continuous point of $H|_{\mcf_F^c}$. Then for all $n\in\ZZ^2$,   $x+n$ is  a H\"older point of $H|_{\mcf_F^c}$ with same H\"older constants and H\"older neighborhood as $x$.
	\end{claim}
	\begin{proof}[Proof of Claim \ref{4 claim Holder Z2}]
		This follows from that the foliation $\mcf_F^c$ and the homeomorphism $H$ are $\ZZ^2$-periodic.
	\end{proof}
	
	\begin{claim}\label{4 claim Holder uleaf}
		Let $x$ be a H\"older continuous point of $H|_{\mcf_F^c}$ and $R>0$. 
		Then $H|_{\mcf_F^c}$ is H\"older continuous at any point $y\in\mcf_F^u(x,R)$ 
		with uniform H\"older constants $\alpha(x,R), C(x,R)$ and uniform size of H\"older neighborhood.
	\end{claim}
	\begin{proof}[Proof of Claim \ref{4 claim Holder uleaf}]
	Note that 
	$$H|_{\mcf_F^c(y)}=  {\rm Hol}^{\mcf_G^u}_{H(x),H(y)}  \circ H|_{\mcf_F^c(x)}\circ {\rm Hol}^{\mcf_F^u}_{y,x}.$$ 
	The conclusion follows from that 
	the holonomy maps along unstable foliations $\mcf_F^u$ and $\mcf_G^u$ are uniformly H\"older continuous 
	(see Proposition \ref{2 prop Holder foliation}).
	\end{proof}
	
	Applying Claim \ref{4 claim Holder Z2} and Claim \ref{4 claim Holder uleaf},  
	we can get the uniformly H\"older continuity of $H|_{\mcf_F^c}$ at any point 
	from the Lipschitz continuity of point $x_0$, 
	by the $u$-accessibility as in the proof of Lemma \ref{4 lem conjugacy}.
	This ends the proof of Lemma \ref{4 lem Holder conjugacy}.
	\end{proof}
	
	Now, we show that the DA map  $f$ is indeed an Anosov map. 
	Notice that $f$ is now conjugate to $g$ via a homeomorphism $h$, 
	it suffices to show that there is $\lambda<0$ such that the center Lyapunov exponent
	\begin{align}
		\lambda^c(p,f)<\lambda,\quad \forall\, p\in {\rm Per}(f).\label{eq. Anosov center exp}
	\end{align}
	Indeed, the conjugacy guarantees that $f$ inherits the transitivity and the shadowing property from $g$. 
	Then, by a standard argument involving the shadowing property with transitivity 
	(see for example \cite[Claim 2.20]{AGGS23} or \cite[Proposition 4.2]{GX2023}), 
	\eqref{eq. Anosov center exp} implies that there is a Riemannian metric $\|\cdot\|$ such that $\log \|Df|_{E^c_f(x)}\|<\frac{\lambda}{2}$, for every point $x\in\TT^2$. 
	Hence, we can conclude that $f$ is Anosov when \eqref{eq. Anosov center exp} holds.
	
	Recall that $h^{-1}|_{\mcf_g^s}$ is H\"older continuous, 
	i.e., there are $0<\alpha<1$, $C>0$ and $\e_0>0$ such that 
	\[
	d_{\mcf_f^c}(h^{-1}(x),h^{-1}(y))\leq C\cdot d_{\mcf_g^s}^\alpha(x,y),
	\quad \forall\, x\in \TT^2\ {\rm and}\ y\in\mcf_g^s(x,\e_0).
	\]
	Since $g$ is Anosov, there is a constant $\lambda_g<0$ such that 
	\begin{align}
		\lambda^s(q,g)<\lambda_g,\quad \forall\, q\in {\rm Per}(g). \label{eq. Anosov stable exp}
	\end{align}
	Let $p\in {\rm Per}(f)$ with period $m\in\NN$ and $x\in \mcf_f^c(p,\delta_0)$.  
	Denote $q=h(p)$ and $y=h(x)$. 
	Then for any $k\in\NN$,
	\begin{align*}
		d_{\mcf_f^c}\big(f^{km}(p),f^{km}(x)\big)&=d_{\mcf_f^c}\big( h^{-1}\circ g^{km}(q), h^{-1}\circ g^{km}(y)     \big)\\
		&\leq C\cdot 	d_{\mcf_g^s}^\alpha\big( g^{km}(q), g^{km}(y)  \big).
	\end{align*}
	By \eqref{eq. Anosov stable exp}, $\lambda^c(p,f)\leq \alpha\cdot\lambda_g$. 
	Let $\lambda=\frac{\alpha\lambda_g}{2}$, 
	we get \eqref{eq. Anosov center exp}, 
	which ends the proof of the $C^1$ part of Theorem \ref{2 thm rigidity accessible}.

\subsection{Proof of the $C^{1+\alpha}$-regularity part of Theorem \ref{2 thm rigidity accessible}}\label{subsec C1+ case}

Let $f$ and $g$ be $C^{r}\ (r>1)$. Then,
the unstable foliations $\mcf^u_F, \mcf^u_G$ of the lifts $F, G:\mrr$ of $f,g$ respectively are uniformly $C^1$-smooth,  see Proposition \ref{2 prop Holder foliation}.
We first show that  $H:\mrr$ the lift of $h$, restricted on each leaf of  $\mcf^c_F$ is  a diffeomorphism. 
	
	Fix $x\in\RR^2$. Recall that the restriction $H|_{\mcf^c_F}: \mcf^c_F(x)\to \mcf^s_G(H(x))$ is monotonic. 
	It follows that there exists a differentiable point $x_0\in \mcf^c_F(x)$ for the map $H|_{\mcf^c_F}$, 
	namely, $D_{x_0}H|_{ \mcf^c_F}$ exists.
	
	\begin{claim}\label{4 claim diff trans}
		For every $R>0$, there exists $C>0$ such that 
		for all $y\in\RR^2$, $n\in\ZZ^2$ and $z\in \mcf^u_F(y+n,R)$, 
		if $y$ is a differentiable point of $H|_{ \mcf^c_F}$, 
		then so is $z$, and $\|D_{z}H|_{ \mcf^c_F}\| < C\|D_{y}H|_{ \mcf^c_F}\|$. 
	\end{claim}
	
	\begin{proof}[Proof of Claim \ref{4 claim diff trans}]
		It is clear that $\|D_yH|_{\mcf^c_F}\|= \|D_{\pi(y)}h|_{\mcf^c_f}\|=\|D_{y+n}H|_{\mcf^c_F}\|$ for every $n\in\ZZ^2$. 
		Consider $z\in\mcf^u_F(y+n,R)$. 
		Since $H$ maps $\mcf^c_F$ and $\mcf^u_F$ to $\mcf^s_G$ and $\mcf^u_G$ respectively 
		(see Proposition \ref{2 prop foliation on R2}), 
		the restriction $H|_{\mcf^c_F}:\mcf^c_F(z) \to \mcf^s_G(H(z))$ satisfies
		\begin{align}
			H(w)={\rm Hol}^u_{G,H(y+n),H(z)}\circ H|_{\mcf^c_F}\circ  {\rm Hol}^u_{F,z,y+n}(w),
			\quad   \forall\, w\in \mcf^c_F(z).\label{eq. 4. 2}
		\end{align}
		Since the holonomy maps of $F$ and $G$ are uniformly $C^1$-smooth, 
		$H|_{\mcf^c_F}$ is differentiable at the point $z$,  
		moreover, there exists $C>0$ depending only on $R$ such that $\|D_{z}H|_{ \mcf^c_F}\|<C\|D_{y}H|_{ \mcf^c_F}\|$. 
	\end{proof}

	By Claim \ref{4 claim diff trans} and the estimation of $u$-accessibility (Theorem \ref{3  thm dichotomy}), 
	we can deduce that the map $DH|_{\mcf^c_F}$ is differentiable
	and the derivative is uniformly bounded away from $0$.
  
	\begin{claim}\label{4 claim diff everywhere}
		The map $H|_{\mcf^c_F}$ is differentiable at every point $x\in \RR^2$ 
		and there exists $C>1$ such that
		$$
		C^{-1}<\|D_xH|_{\mcf^c_F}\|<C,
		\quad\forall\, x\in \RR^2.
		$$
	\end{claim}
	
	\begin{proof}[Proof of Claim \ref{4 claim diff everywhere}]
		Let $x_0\in\RR^2$ be a differentiable point of $H|_{\mcf^c_F}$. 
		Let constants $R>0, N\in\NN$ be given by the second part of Theorem \ref{3  thm dichotomy}.
		
		Let $y\in\RR^2$. By Theorem \ref{3  thm dichotomy}, 
		there exist at most $2N$ unstable manifolds connecting points $\pi(x_0)$ and $\pi(y)$. 
		In particular, these unstable manifolds can be chosen as projections of the unstable manifolds on $\RR^2$.  
		Then, it follows that there exists $k\leq 2N$ 
		with $x_0, x_1,\dots, x_k\in\RR^2$ and $m,n_0, n_1, n_2,\dots, n_k\in\ZZ^2$ 
		such that $x_{i+1}\in\mcf^u_F(x_i+n_i,R)$ for $0\leq i\leq k-1$, 
		and $y+m\in \mcf^u_F(x_k+n_k,R)$.
		
		Applying Claim \ref{4 claim diff trans} to $x_i\ (1\leq i\leq k)$, 
		$H|_{\mcf^c_F}$ is differentiable at points $x_i$ and $y+m$.  
		Moreover, there exists $C_0=C_0(R)>1$ such that 
		\[  C_0^{-1}\|D_{x_i}H|_{\mcf^c_F}\|<\|D_{x_{i+1}}H|_{\mcf^c_F}\|<C_0\|D_{x_i}H|_{\mcf^c_F}\|, \] 
		and 
		\[ C_0^{-1}\|D_{x_k}H|_{\mcf^c_F}\|<\|D_{y+m}H|_{\mcf^c_F}\|=\|D_yH|_{\mcf^c_F}\|<C_0\|D_{x_k}H|_{\mcf^c_F}\|. \]
		Let $C_1=(C_0)^{2N}$. 
		Then by induction, 
		\begin{align}
			C_1^{-1}\|D_{x_0}H|_{\mcf^c_F}\|<\|D_yH|_{\mcf^c_F}\|<C_1\|D_{x_0}H|_{\mcf^c_F}\|,
			\quad \forall\, y\in\RR^2. \label{eq. 4.1}
		\end{align}
		
		Note that \eqref{eq. 4.1} implies $\|D_{x_0}H|_{\mcf^c_F}\|\neq 0$. 
		Otherwise, for every point $y\in\mcf^c_F(x_0)$, $\|D_yH|_{\mcf^c_F}\|=0$. 
		This contradicts the fact that $H|_{\mcf^c_F(x_0)}$ is surjective. 
		Take $C>\max\{ C_1\|D_{x_0}H|_{\mcf^c_F}\|, C_1\|D_{x_0}H|_{\mcf^c_F}\|^{-1}\}$, 
		which is a constant satisfying Claim \ref{4 claim diff everywhere}.
	\end{proof}
	
	By Claim \ref{4 claim diff everywhere}, 
	there is  point  $z\in\mcf^c_f(x)$ for some point $x\in \TT^2$ such that $Dh|_{\mcf^c_f}$ is continuous at $z$; 
	see for example \cite[Theorem 7.3]{O1980}. 
	Hence, by the $u$-accessibility and the uniformly $C^1$-smooth unstable foliations,  
	we get that $Dh|_{\mcf^c_f}$ is actually continuous at every point in $\TT^2$. 
	By Claim \ref{4 claim diff everywhere} again,  
	$h$ restricted to each center leaf is a diffeomorphism. 
	In particular, $f$ is Anosov and $h$ is a conjugacy.  Moreover, since $h$ is $C^1$-smooth along stable manifolds, $\lambda^s(p,f)=\lambda^s(h(p),g)$, for all $p\in{\rm Per}(f)$. By classical rigidity argument of \cite{dL92}, this periodic data  implies that $h|_{\mcf^s_f}$ is $C^r$-smooth. 
	It  completes the proof of $C^{1+\alpha}$-part of Theorem \ref{2 thm rigidity accessible}.

\section{DA maps without Anosov factors}\label{sec DA without Anosov factors}

In this section, we prove Theorem \ref{main-thm-factor}, 
i.e., there is a $C^1$-open subset $\mathcal{U}$ in the space of maps on $\TT^2$ such that 
every map in $\mathcal{U}$, homotopic to Anosov, admits no Anosov factor. 

First, we construct an example which originates from \cite[Section 17.2]{KH95}. 

\begin{prop}\label{5 prop example}
	Let $A:\mtt$ be a linear Anosov map. 
	There is a $C^\infty$-smooth DA map $f:\mtt$ homotopic to $A$ 
	such that $f$ is not Anosov. 
	In particular, $f$ is not semi-conjugate to any Anosov map. 
	Moreover, $f$ indeed admits a source, i.e., a fixed point with the positive center Lyapunov exponent.
\end{prop} 

\begin{proof}
	Without loss of generality, we assume that the eigenvalues of $A$ are $0<\mu_s<1<\mu_u$. 
	Note that $\mu_s\cdot\mu_u={\rm deg}(A)\geq 2$. 
	Let $T\TT^2=L^s\oplus L^u$ be the hyperbolic splitting of $A$. 
	Up to a change of Riemannian metric, 
	we can assume that $L^s$ and $L^u$ are orthogonal, and $A$ can be represented by
    \[A(x,y)=(\mu_s x, \mu_u y).\]
	Let $0=(0,0)$ be a fixed point of $A$, and $U$ be a small neighborhood of $0$. 
	
	Let $V=(-r,r)^2$ be a rectangle such that the closure is also contained in $U$. 
	Take a $C^\infty$ bump function $\phi$ supported on $(-r,r)$ such that 
	$0\leq \phi\leq 1$, $\phi(y)=1$ for $y$ in a neighborhood of $0$, and $\phi'(0)=0$.

	Consider $1<\mu_c<\mu_u$, $K\triangleq\frac{\mu_c}{\mu_s}-1>0$, and $0<\delta<1$. 
	Let $\eta:\RR\to \RR$ be a $C^\infty$-smooth compactly supported function such that
	\[ 	
	\eta(0)=0,	\quad \eta'(0)=K	
	\quad {\rm and}\quad 	
	-\delta<\eta'(x)\leq K,\ \forall\, x\in\RR.	
	\]
	For \(\e>0\), let 
	\[
	\eta_\e(x) :=\e\cdot\eta\left(\frac{x}{\e}\right).	
	\]
	Then $\eta_\e(0)=0$, $\eta_\e'(0)=K$, $-\delta<\eta_\e'(x)\leq K$, ${\rm supp}(\eta_\e)=\e\cdot{\rm supp}(\eta)$, 
	and $\|\eta_\e\|_{C^0} \to 0$ as $\e\to  0$.  
	Moreover, take $\e$ small such that ${\rm supp}(\eta_\e)\subset (-r,r)$. 
	
	Now, we define $f=f_{\e}$ as follows
	\[ f(x,y)=A(x,y)+\big(\mu_s\phi(y)\eta_\e(x),0\big)=\big(\mu_s x+\mu_s\phi(y)\eta_\e(x),	\mu_u y\big).\]
	Then 
	\[	D_{(x,y)}f =
	\begin{pmatrix}
		P_\e(x,y)&Q_\e(x,y)\\
		0&\mu_u
	\end{pmatrix}, 
	\]
	where \[P_\e(x,y)=	\mu_s\big(1+\phi(y)\eta_\e'(x)\big)\quad {\rm and} \quad Q_\e(x,y)	= \mu_s\phi'(y)\eta_\e(x).\]
	Since $\|\eta_\e\|_{C^0}
	\to0$ as $\e\to0$ and $-\delta <\eta_\e'(x)\leq K$, we have 
	\begin{align}
		\|Q_\e\|_{C^0}	\leq	\mu_s\|\phi'\|\cdot\|\eta_\e\|_{C^0}
		\to 0  \quad {\rm and}\quad \mu_s(1-\delta) \leq P_\e(x,y) \leq\mu_s(1+K) =\mu_c. \label{eq. example}
	\end{align}
	This implies $\det D_{(x,y)}f =\mu_uP_\e(x,y) \geq	\mu_u\mu_s(1-\delta)>0$. 
	Thus, $f$ is a $C^\infty$ local diffeomorphism.  
	
	Note that $f$ is homotopic to $A$, 
	by replacing $\mu_s\phi(y)\eta_\e(x)$ 
	by $t\cdot\mu_s\phi(y)\eta_\e(x)$ with $0\leq t\leq 1$ in the construction of $f$, 
	we indeed get an isotopy from $A$ to $f$. 
	It is clear that $0=(0,0)$ is also a fixed point of $f$ with two Lyapunov exponents $\log \mu_u>\log\mu_c>0$, 
	and hence a source for $f$. 
	In particular, $f$ is not Anosov.
		
	The rest is to prove that $f$ admits a partially hyperbolic splitting with a unstable bundle. 
	Then $f$ is a DA map but not Anosov. 
	By Theorem \ref{main-thm-conjugacy}, $f$ admits no Anosov factor.
	
	To show the partial hyperbolicity of $f$, 
	it suffices to prove that 
	there is a $Df$-invariant unstable cone-field 
	(this sufficiency is well-known and we refer to \cite{GX2023} for more details). 
	Precisely, we show that	there are constants $a>0$, $\kappa>1$ and a family of cones
	\[	\mathcal{C}_a^u(z):=
	\left\{ 	v=v_s+v_u\in L^s(z)\oplus L^u(z),\ 
	\| v_u\|\geq a\| v_s\|	\right\}	\]
	such that for all $z\in\TT^2$ and $v\in\mathcal{C}_a^u(z)$,
	\begin{align}
		Df_z\big(\mathcal{C}_a^u(z)\setminus\{0\}\big)
		\subset	{\rm Int}\Big( \mathcal{C}_a^u(f(z))\Big)\quad {\rm and}\quad \| Df_zv\|\geq\kappa\| v\|.
		\label{eq. example 1}
	\end{align}

	In the following, we prove \eqref{eq. example 1}.   
	Let $a>0$ be large enough such that
	\[\kappa :=\frac{\mu_u a}{\sqrt{1+a^2}} > \mu_c,\] 
	and let $\e>0$ be small enough such that
	\begin{align}
		a\| Q_\e\|_{C^0}<\mu_u-\mu_c.\label{eq. example 2}
	\end{align}
	Then $a,\kappa$ satisfy \eqref{eq. example 1} 
	with respect to the map $f=f_{\e}$ where we fix a small $\e>0$.

	Indeed, let $v=(v_s,v_u)\in\mathcal{C}_a^u(z)\setminus\{0\}$. 
	It is clear that $f(z)=A(z)$, for any $z\in \TT^2\setminus V$. 
	Thus, \eqref{eq. example 1} holds for all $z\in \TT^2\setminus V$ and $v\in \mathcal{C}_a^u(z)$.  
	Now, we take $z\in V$. Then,
	\[ Df_zv = \big( P_\e(z)v_s+Q_\e(z)v_u,\mu_uv_u\big).\]
	Since $ \| v_u\|\geq a \| v_s\|$,  
	it follows from \eqref{eq. example} that
	\begin{align*}
			\| P_\e(z)v_s+Q_\e(z)v_u	\|
		\leq	|P_\e(z)|\cdot \|v_s\|+	|Q_\e(z)|\cdot\|v_u\| 	
		\leq	\left(	\frac{\mu_c}{a}	+	\| Q_\e\|_{C^0}	\right)	\| v_u\|.
	\end{align*}
	Therefore, by \eqref{eq. example 2},
	\begin{align*}
		a\| P_\e(z)v_s+Q_\varepsilon(z)v_u\|<	\mu_u\| v_u\|.
	\end{align*}
	Thus, $Df_zv\in{\rm Int}\big(\mathcal{C}_a^u(f(z))\big)$.  
	Moreover, since $\|v\|^2=\|v_s\|^2+\|v_u\|^2\leq \frac{1}{a^2}\| v_u\|^2+\| v_u\|^2$,
	\[ \| v_u\| \geq \frac{a}{\sqrt{1+a^2}}\|v\|. \]
	Note that $\mu_uv_u$ is a coordinate component of $Df_zv$,  
	therefore
	\[ \| Df_zv\| \geq \mu_u\lVert v_u\| \geq
	\frac{\mu_ua}{\sqrt{1+a^2}}\|v\| = \kappa\|v\|. \]
	As a result, \eqref{eq. example 1} holds for all $z\in\TT^2$ and $v\in\mathcal{C}_a^u(z)$. 
	This completes the proof of Proposition \ref{5 prop example}.
	\end{proof}

Finally, by the $C^1$-openness of DA maps, we obtain Theorem \ref{main-thm-factor}.

\begin{proof}[Proof of Theorem \ref{main-thm-factor}]
	For any linear Anosov map $A$ on $\TT^2$, by Proposition \ref{5 prop example}, 
	there is a DA map $f$ homotopic to $A$ without any Anosov factor. 
	Let $\mathcal{U}_f$ be a $C^1$ neighborhood of $f$ in the space of local diffeomorphisms. 
	When $\mathcal{U}$ is small enough, 
	every map  $g\in \mathcal{U}_f$ is homotopic to $A$ and partially hyperbolic with a unstable bundle 
	(with respect to orbits), namely, $g$ is a DA map. 
	One can check the $C^1$-openness of partial hyperbolicity by lifting to the universal cover $\RR^2$ 
	(see \cite{ManePugh1975}), or check the persistence of unstable cone-fields. 
	Moreover, $f$ admits a source, and so does its $C^1$-small perturbation $g$. 
	This implies that $g$ is DA but not Anosov. 
	Thus, by Theorem \ref{main-thm-conjugacy}, $g$ cannot be semi-conjugate to any Anosov map.
\end{proof}

\section*{Acknowledgement}
The authors would like to thank Yi Shi for suggesting this rigidity project, 
which is motivated by a question of J\'{e}r\^{o}me Buzzi 
at the conference of Beyond Uniform Hyperbolicity (2023).  
R. Gu was partially supported by NSFC (Nos. 12571203, 12601354);
M. Xia was partially supported by NSFC (12501238) and
the Fundamental Research Funds for the Central Universities (DUT24RC(3)112).

\appendix

\section{A special leaf guarantees the semi-conjugacy}
\label{sec-app-a}

In this appendix, we give the proof of  Proposition \ref{2 prop special implies semi-conjugacy}.  
Let $f$ be a  DA map on $\TT^2$ with linearization $A$. 
Let $F:\mrr$ be a lift of $f$ and $H:\mrr$ be a semi-conjugacy from $F$ to $A$.  
In Proposition \ref{2 prop special implies semi-conjugacy}, 
we assume that there is a periodic special unstable manifold (provided by Lemma \ref{3 lemma xn}).  
This assumption will make the proof simpler. Here, we establish the result for a more general case.

\begin{prop}\label{5 prop special to semi-conj}
	If there exists a point $x_0\in\RR^2$ such that $\mcf^u_F(x_0+n)=\mcf^u_F(x_0)+n$ for all $n\in\ZZ^2$, 
	then $f$ is semi-conjugate to $A$. 
\end{prop}

Note that Lemma \ref{3 lem x_0} already gives the assumption of Proposition \ref{5 prop special to semi-conj}.

\begin{proof}[Proof of Proposition \ref{5 prop special to semi-conj}]
We construct a ``big'' set $\mathbf{F}$ foliated by unstable leaves.
\begin{align}
	\mathbf{F}:=\bigcup_{k\in\NN, n\in\ZZ^2} \mcf^u_F\big(F^k(x_0)+A^kn\big). \label{eq. prop 3.4 set kn}
\end{align}
Since $x_0$ satisfies the assumption, for all $k\in\NN$ and $n\in\ZZ^2$,
\begin{align}
	\mcf_F^u\big(F^k(x_0)+A^kn\big)=F^k\big(\mcf^u_F(x_0)+n\big)=\mcf_F^u\big(F^k(x_0)\big)+A^kn.\label{eq. prop 3.4 kn}
\end{align}
 For $k\in\NN$, let 
\[\mathbf{F}_0:=\bigcup_{n\in\ZZ^2} \mcf^u_F(x_0+n)=\mcf^u_F(x_0)+\ZZ^2,\quad \mathbf{F}_k:=\bigcup_{n\in\ZZ^2} \mcf^u_F\big(F^k(x_0)+A^kn\big)=\mcf^u_F\big(F^k(x_0)\big)+A^k\ZZ^2.\]

 When $\pi(\mcf^u_F(x_0))$ is a periodic leaf,   we will denote by $x_0$, the unique periodic point of $f$ on $\pi(\mcf^u_F(x_0))$.  Then, we can  further assume  that $x_0$ is a  fixed point. We note that in this case, the sets $\mathbf{F}$ and $\mathbf{F}_0$ coincide.  This will make the following proof easier. But keep in mind, we consider the general case here.

\begin{lem}\label{5 lem Finvariant}
	The closure $\bar{\mathbf{F}}$ is  $F$-invariant, and so is $\bar{\mathbf{F}}_0$
\end{lem}

\begin{proof}[Proof of Lemma \ref{5 lem Finvariant}]
	By definition, $F(\bar{\mathbf{F}})\subset \bar{\mathbf{F}}$ and 
	$\mathbf{F}\setminus F({\mathbf{F}})= \mcf^u_F(x_0)+\ZZ^2=\mathbf{F}_0$.
	Note that there exist $n_i\in\ZZ^2 \ (1\leq i\leq l:={\rm deg}(A))$ such that 
	$\ZZ^2= \bigcup_{1\leq i\leq l}(n_i+A\ZZ^2)$, and hence
	\begin{align}
    \mathbf{F}_0=\bigcup_{1\leq i\leq l} \big(\mcf^u_F(x_0+n_i)+A\ZZ^2\big). \label{eq. 5. F0}
	\end{align}
	Moreover, by \eqref{eq. prop 3.4 kn},
	$$\bigcup_{n\in\ZZ^2} \mcf^u_F(F(x_0)+An)=\mcf^u_F\big(F(x_0)\big)+A\ZZ^2\subset  F({\mathbf{F}}).$$
	To prove $\bar{\mathbf{F}}$ being $F$-invariant, 
	by \eqref{eq. 5. F0},
	it suffices to show that for each $1\leq i\leq l$, 
	\begin{align}
		\overline{\mcf^u_F(x_0+n_i)+A\ZZ^2} \subset \overline{\mcf^u_F\big(F(x_0)\big)+A\ZZ^2}.  
		\label{eq. 6. inclusion}
	\end{align}
	
\begin{claim}\label{5 claim gamma inclusion}
	Let $\Gamma=\big\{x\in\RR^2\ |\ H^{-1}\circ H(x)=\{x\} \big\}$ be the set of $H$-injection points.
	Then 
	\[
	\bar{\Gamma}\subset \overline{\bigcup_{n\in\ZZ^2}\mcf^u_F(y+A^kn)}, 
	\quad \forall\, y\in\RR^2, k\in\NN.  
	\]
\end{claim}
\begin{proof}[Proof of Claim \ref{5 claim gamma inclusion}]
	The proof is the same as Claim \ref{3 claim closure}. 
	Replace \eqref{3 eq. closure} by the following:  
	the set
	\[
	H\big(\bigcup_{n\in\ZZ^2}\mcf^u_F(y+A^kn) \big)
	=\bigcup_{n\in\ZZ^2}\tildeL^u\big(H(y+A^kn)\big)
	\]
	is dense in $\RR^2$. 
	And this is just a corollary of 
	Proposition \ref{2 prop semiconj on stable} and Proposition \ref{2 prop minimal foliation}.
\end{proof}

	By Claim \ref{5 claim gamma inclusion} and \eqref{eq. prop 3.4 kn}, 
	for every $1\leq i\leq l$,
	$$\emptyset \neq \bar{\Gamma}\subset \overline{\mcf^u_F(x_0+n_i)+A\ZZ^2} \ \bigcap\  \overline{\mcf^u_F\big(F(x_0)\big)+A\ZZ^2}.$$
	Let $x\in \bar{\Gamma}$. There are  sequences $\{m^1_k\}_{k\in\NN}$ and $\{m^2_k\}_{k\in\NN}$ with $m^j_k\in A\ZZ^2$ for all $k\in\NN$ and $j=1,2$ such that 
	\[ \mcf^u_F(x_0+n_i)+m^1_k\to \mcf^u_F(x) \quad {\rm and} \quad \mcf^u_F\big(F(x_0)\big)+m^2_k\to \mcf^u_F(x),\]
	as $k\to+\infty$.
	It follows that 
	$ \mcf^u_F\big(F(x_0)\big)+m^2_k-m^1_k\to  \mcf^u_F(x_0+n_i)$.
	Since $(m^2_k-m^1_k)\in A\ZZ^2$, 
	$$\mcf^u_F(x_0+n_i)\subset  \overline{\mcf^u_F\big(F(x_0)\big)+A\ZZ^2}.$$
	Hence this proves \eqref{eq. 6. inclusion}. 
 
	To prove the $F$-invariance of $\bar{\mathbf{F}}_0$, we show $\bar{\mathbf{F}}_0=\bar{\mathbf{F}}$. 
	By Claim \ref{5 claim gamma inclusion} and above argument, 
	we have  $\bar{\mathbf{F}}_k\subset \bar{\mathbf{F}}_0$ for every $k\in\NN$.  
	Indeed, let $x\in \bar{\Gamma}\subset \bar{\mathbf{F}}_k\cap \bar{\mathbf{F}}_0$. 
	There are $n^1_j\in\ZZ^2$ and $n^2_j\in A^k\ZZ^2$ such that  as $j\to+\infty$,
	\[ \mcf^u_F(x_0)+n^1_j\to \mcf^u_F(x) \quad {\rm and} \quad \mcf^u_F\big(F^k(x_0)\big)+n^2_j\to \mcf^u_F(x).\]
	Hence $\mcf^u_F(x_0)+n^1_j-n^2_j \to  \mcf^u_F\big(F^k(x_0)\big)$ and $\bar{\mathbf{F}}_k\subset \bar{\mathbf{F}}_0$.
	Since $\mathbf{F}=\cup_{k\in\NN}\mathbf{F}_k$ and $\bar{\mathbf{F}}_k\subset \bar{\mathbf{F}}_0$, 
	we have $\bar{\mathbf{F}}=\bar{\mathbf{F}}_0$.
	
	This completes the proof of Lemma \ref{5 lem Finvariant}.
\end{proof}
	
	For any $k\in\NN$ and $n\in\ZZ^2$, define  
	\[ H_{k,n}:\mcf_F^u\big(F^k(x_0)+A^kn\big) \to \tildeL^u\big(A^k\circ H(x_0)+A^kn\big)\] 
	by $H_{k,n}(x)=H(x-A^kn)+A^kn$.
	By \eqref{eq. prop 3.4 kn}, the map $H_{k,n}$ is well-defined.  
	
\begin{claim}\label{5 claim barH}
	If $x\in \mcf_F^u\big(F^{k_1}(x_0)+A^{k_1}n_1\big) \cap \mcf_F^u\big(F^{k_2}(x_0)+A^{k_2}n_2\big) $ 
	for some $k_1\geq k_2\geq 0$ and $n_1, n_2\in\ZZ^2$ with $k_1\neq k_2$ or $n_1\neq n_2$, 
	then $H_{k_1,n_1}(x)=H_{k_2,n_2}(x)$.
\end{claim}
\begin{proof}[Proof of Claim \ref{5 claim barH}]
	Note that the assumption of this claim can only happen in the case of fixed point $x_0$. 
	Indeed, if $\mcf_F^u\big(F^{k_1}(x_0)+A^{k_1}n_1\big) \cap \mcf_F^u\big(F^{k_2}(x_0)+A^{k_2}n_2\big)\neq \emptyset$, 
	by \eqref{eq. prop 3.4 kn},   
	$x_0+n_2-A^{k_1-k_2}n_1\in \mcf^u_F\big(F^{k_1-k_2}(x_0) \big)$. 
	Then  $\pi (\mcf^u_F(x_0))$ is an $f$-periodic leaf. 
	By the assumption above, $x_0$ is an $F$-fixed point.  
	
	In the case that $x_0$ is a fixed point and $f$ is special, 
	Claim \ref{5 claim barH} has been proved in \cite[Lemma 3.1]{GX2023}. 
	Here we replace the special property of $f$ by the one of a leaf $\mcf^u_F(x_0)$ 
	and list the key points of the proof. 	
	Since $x_0$ is fixed, we consider the case that $k_1=k_2=0$. 
	Let $x\in \mcf^u_F(x_0+n_1)\cap \mcf^u_F(x_0+n_2)$. 
	Then 
	\begin{enumerate}[label=(\roman*)]
		\item $H_{0,n_1}(x), H_{0,n_2}(x)\in \tildeL^s(H(x))$; 
		See the first item of Proposition \ref{2 prop semiconj on stable}.
		
		\item  $x_0+l(n_1-n_2)\in \mcf^u_F(x_0)$, for all $l\in \ZZ$. 
		This follows from $\mcf^u_F(x_0+n)=\mcf^u_F(x_0)+n$ for all $n\in\ZZ^2$.
		
		\item $H(x_0)+n_1-n_2\in \tildeL^u(H(x_0))$. 
		This follows from the item (ii). 
		Indeed, by $\|H-{\rm Id}_{\RR^2}\|_{C^0}<C_0$ and $x_0+l(n_1-n_2)\in \mcf^u_F(x_0)$,  
		$$H(x_0)+l(n_1-n_2)\in B_{2C_0}\big(\tildeL^u(H(x_0))\big).$$ 
		If $H(x_0)+n_1-n_2\notin \tildeL^u(H(x_0))$, 
		then $d\big(H(x_0)+l(n_1-n_2),  \tildeL^u(H(x_0) \big)$ is larger than $2C_0$ 
		by taking $l$ large enough. 
		It is a contradiction.
	\end{enumerate}
    Now we have 
    \begin{align*}
    	H_{0,n_2}(x)=H(x-n_2)+n_2 & \in H\big( \mcf^u_F (x_0+n_1-n_2) \big)+n_2,
    	\quad (x\in \mcf^u_F(x_0+n_1))\\
    	&=H\big( \mcf^u_F (x_0) \big)+n_2=\tildeL^u(H(x_0))+n_2, 
    	\quad (\text{the item (ii)})\\
    	& =\tildeL^u\big(H(x_0)+n_1-n_2\big)+n_2, 
    	\quad (\text{the item (iii)})\\
    	& = \tildeL^u\big( H(x_0)+n_1  \big)=\tildeL^u\big( H_{0,n_1}(x) \big). 
    \end{align*}
It follows from item (i) that $H_{0,n_2}(x)=\tildeL^u\big( H_{0,n_1}(x) \big)\cap \tildeL^s\big(H(x)\big)=H_{0,n_1}(x)$.
\end{proof}

By Claim \ref{5 claim barH}, the map
\[ 
\overline{H}:	\bigcup_{k\in\NN, n\in\ZZ^2} \mcf^u_F\big(F^k(x_0)+A^kn\big) \to \bigcup_{k\in\NN, n\in\ZZ^2}\tildeL^u\big(A^k\circ H(x_0)+A^kn\big)
\] 
given by $\overline{H}(x)=H_{k,n}(x)$, for any $x\in \mcf^u_F\big(F^k(x_0)+A^kn\big)$, is well-defined.

\begin{lem}\label{5 lem H property}
	The map $\overline{H}$ (defined on $\mathbf{F}$) satisfies:
	\begin{enumerate}
		\item There exists $C>0$ such that $\| \overline{H}-{\rm Id}_{\mathbf{F}} \|_{C^0}<C$;
		\item $\overline{H}\circ F(x)= A\circ \overline{H}(x)$, for all $x\in\mathbf{F}$;
		\item $\overline{H}(x+n)= \overline{H}(x)+n$, for all $x\in \mathbf{F}_0$ and $n\in\ZZ^2$;
		\item $\overline{H}$ is uniformly continuous on $\mathbf{F}_0$; 
		\item  $\overline{H}$ uniquely  continuously extends  to $\bar{\mathbf{F}}_0$, 
			moreover $\overline{H}|_{\bar{\mathbf{F}}_0}=H|_{\bar{\mathbf{F}}_0}$.
	\end{enumerate}
\end{lem}

\begin{proof}[Proof of Lemma \ref{5 lem H property}]
	
	This lemma is adapted from \cite[Lemma 3.3]{GX2023} 
	where the authors proved it in the case that $x_0$ is a fixed point.  
	For completeness, we briefly prove this lemma here.  
	
	For the first two items, 
	let $x\in \mcf^u_F\big(F^k(x_0)+A^kn\big)$.  
	Then $\overline{H}(x)=H_{k,n}(x)=H(x-A^kn)+A^kn$. Hence 
	$d\big(\overline{H}(x),x\big)=d\big( H(x-A^kn), x-A^kn\big)<C$,
	where $C$ is given by Proposition \ref{2 prop semi-conj in R2}.
	So the first item holds. 
	Moreover,
	$F(x)\in \mcf_F^u\big(F\big(F^{k}(x_0)+A^{k}n\big)\big)=\mcf_F^u\big(F^{k+1}(x_0)+A^{k+1}n\big)$,
	and 
	\begin{align*}
		\overline{H}\circ F(x)=H_{k+1,n}(F(x))&=H\big(F(x)-A^{k+1}n   \big) +A^{k+1}n
		=H\circ F(x-A^kn)+A^{k+1}n\\&=A\circ H(x-A^kn)+A^{k+1}n
		=A\circ (H(x-A^kn)+A^kn)=A\circ\overline{H}(x).
	\end{align*}
	Hence the second item holds. 
	
	For the third item, let $x\in \mathbf{F}_0$, namely, 
	there exists $n_x\in\ZZ^2$ such that $x\in\mcf^u_F(x_0+n_x)$.  
	For any $n\in\ZZ^2$, by the choice of $x_0$,
	$x+n\in \mcf^u_F(x_0+n_x)+n=\mcf^u_F(x_0+n_x+n)$.
	Hence
	\[
	\overline{H}(x+n) = H\big(x+n-(n_x+n)\big)+(n_x+n)
	=\big(H(x-n_x)+n_x\big)+n=\overline{H}(x)+n.
	\]
	
	Then we prove the forth item. The proof is the same as \cite{GX2023}.
	By the definition and the uniform continuity of $H$ on $\RR^2$, 
	$\overline{H}$ is uniformly continuous along each leaf of $\mcf^u_F(x)\subset \mathbf{F}_0$.   
	Then, it suffices to prove the uniform continuity of $\overline{H}|_{\mathbf{F}_0}$ along each leaf of $\mcf^c_F$. 
	By contradiction, we assume that 
	there exist $\e_0>0$, $\delta_t> 0$ with $\delta_t\to 0$, $n_t\in\ZZ^2$, $m_t\in\ZZ^2$ 
	and points $x_t\in \mcf^u_F(x_0+n_t)$, $y_t\in \mcf^u_F(x_0+m_t) \cap \mcf^c_F(x_t)$ 
	such that
	$d(x_t,y_t)<\delta_t$ and $d\big(\overline{H}(x_t),\overline{H}(y_t)\big)\geq\e_0$.
	For given $\e_0>0$, there exists $\eta_0$ such that 
	if $d(x,y)\geq\e_0$ and $y\in \tildeL^s(x)$, 
	then $d\big( \tildeL^u(x),\tildeL^u(y) \big)\geq \eta_0$.
	Let $C>0$ be taken as above, i.e., $\|H-{\rm Id}_{\RR^2}\|_{C^0}<C$.
	For this $\eta_0$ and $C$, let $l\in\NN$ such that $l\eta_0>3C$. 
	Then take $\delta_t\to 0$ small enough such that $l\delta_t<C$. 
	By the first item of Proposition \ref{2 prop semiconj on stable},  
	$\overline{H}$ maps $\mcf^c_F$ to $\tildeL^s$. 
	Since $d\big(\overline{H}(x_t),\overline{H}(y_t)\big)\geq\e_0$, 
	we have
	\[ d\big( \tildeL^u(H(x_0)) +n_t-m_t\ ,\ \tildeL^u(H(x_0))  \big) =
	d\big( \tildeL^u(H(x_0)+n_t)\ ,\ \tildeL^u(H(x_0)+m_t)      \big)
	\geq \eta_0 .\]
	Hence, for $l\in\NN$ given as above, 
	\begin{align}
		d\big(\tildeL^u(H(x_0)) +l(n_t-m_t)\ ,\ \tildeL^u(H(x_0))     \big) \geq l\eta_0>3C . 
		\label{eq. 3. to infinity}
	\end{align}
	On the other hand, since $d(x_t,y_t)<\delta_t$, we have 
	\begin{align}
			d\big( \mcf^u_F(x_0+n_t)\ ,\ \mcf^u_F(x_0+m_t)   \big) 
		=  d\big( \mcf^u_F(x_t)\ ,\ \mcf^u_F(y_t)  \big) <\delta_t. \label{eq. A.1}
	\end{align}
  This deduces that	for $l\in\NN$ given above, 
  \begin{align}
  		d\big( \mcf^u_F(x_0)+l(n_t-m_t),\mcf^u_F(x_0)  \big)  <l\delta_t. \label{eq. A.2} 
  \end{align}
See \cite[Formula(3.4)]{GX2023} for  a complete proof of \eqref{eq. A.1} implying \eqref{eq. A.2}. 
	Since $\|H-{\rm Id}_{\RR^2}\|_{C^0}<C$,  we have
	\begin{align}
		d\big( \tildeL^u(H(x_0))+l(n_t-m_t)\ ,\ \tildeL^u(H(x_0))  \big)  <l\delta_t+2C<3C.
		\label{eq. 3. to small 2}
	\end{align}
	Then, \eqref{eq. 3. to small 2} contradicts \eqref{eq. 3. to infinity}.  
	
	Finally, we prove the fifth item.   
	By the forth item, the map $\overline{H}$ can be extended to $\bar{\mathbf{F}}_0$. 
	By the first item, we can assume that $\|H-{\rm Id}_{\RR^2}\|_{C^0}\leq C$ 
	and $\|\overline{H}-{\rm Id}_{\bar{\mathbf{F}}_0}\|_{C^0}\leq C$. 
	By the second item,
	$\overline{H}\circ F(x)= A\circ \overline{H}(x)$, for all $x\in \bar{\mathbf{F}}_0$.
	Recall that $\bar{\mathbf{F}}_0$ is $F^{\pm}$-invariant (Lemma \ref{5 lem Finvariant}),
	hence for every $x\in\bar{\mathbf{F}}_0$ and $k\in\ZZ$,
	\begin{align}
		d\big(A^k\circ H(x),A^k\circ\overline{H}(x)\big)= d\big(H\circ F^k(x),
		\overline{H}\circ F^k(x)\big) \leq 2C,
		\label{eq. 3. barH 1}
	\end{align}
	It follows from \eqref{eq. 3. barH 1} that $H(x)=\overline{H}(x)$ for every $x\in\bar\Gamma$.
	
	This ends the proof of Lemma \ref{5 lem H property}.
\end{proof}

By Claim \ref{5 claim gamma inclusion}, 
the closure of the set of  $H$-injection points has  $\bar{\Gamma}\subset \bar{\mathbf{F}}_0$. 
Then applying Lemma \ref{5 lem H property}, 
one can get the following property whose proof is exactly the same as \cite[Proposition 3.6]{GX2023}.

\begin{lem}\label{5 lem semi-h}
The set $\Gamma$ satisfies that $\bar{\Gamma}+\ZZ^2=\bar{\Gamma}$ and
	 $H(x+n)=H(x)+n$, for any $x\in \bar{\Gamma}$ and $ n\in\ZZ^2$.
\end{lem}

By Lemma \ref{5 lem semi-h} and  $H(\bar{\Gamma})=\RR^2$, 
$H(x+n)=H(x)+n$ for $x\in \RR^2\setminus\bar{\Gamma}$ 
and hence (by Lemma \ref{5 lem semi-h} again) for $x\in\RR^2$, see \cite[Corollary 3.7]{GX2023}. This completes the proof of Proposition \ref{5 prop special to semi-conj}.
\end{proof}

\section{DA local diffeomorphisms with stable bundles}\label{sec sc DA}
\label{sec-app-b}

In this appendix, we consider the $s$-DA maps for completeness of the study on $\TT^2$.
Note that a DA local diffeomorphism with a stable bundle is called an $s$-DA map here.  We notice that the DA maps in Theorem \ref{main-thm-conjugacy} 
and the $s$-DA map have different rigidity phenomenon in non-invertible setting, 
see for example \cite{GX2023}. Nonetheless,
Proposition \ref{2 prop foliation on R2} also holds for an $s$-DA map $f:\mtt$ 
where one can replace $\mcf^u_F$ by $\mcf^s_F$, the stable foliation of the lift $F:\mrr$
(see \cite{GX2023}). 
Then, $F$ admits invariant splitting $E^s_F\oplus E^c_F$, and the center bundle $E^c_F$ is integrable.

First, we also have a dichotomy on the accessibility as Theorem \ref{3  thm dichotomy}. 

\begin{defn}\label{2 def c-accessible}\label{2 def c-acc}
	Let $f:\mtt$ be an $s$-DA map and $F:\mrr$ be a lift. 
	We say that $f$ is \textit{$c$-accessible}, 
	if for any $x,y\in \TT^2$, 
	there exist points $y_i \ (0\leq i\leq k)$ with $y_0=x, y_k=y$ and lifts $y_i'\in\RR^2$ of $y_i$ 
	such that  
	\[ y_{i+1}\in \pi\big(\mcf_F^c(y_i')\big), \quad \forall \ 0\leq i\leq k-1.\]
\end{defn}

The following result is adapted from Theorem \ref{3  thm dichotomy}. 
Different from the previous case of DA maps, 
its proof does not need an equivalent description of the special property 
stated in Theorem \ref{2 thm rigidity special}. 

\begin{prop}\label{7  prop dichotomy}
	Let $f:\mtt$ be a $C^1$-smooth  $s$-DA map. Then one has the following dichotomy:
	\begin{itemize}
		\item either $f$ is special.
		
		\item or $f$ is $c$-accessible. 
		Moreover, there exist constants  $R>0$ and $N\in\NN$ such that 
		any two points $x,y\in\TT^2$ can be connected by at most $N$ center manifolds with size no more than $R$. 
	\end{itemize}  
\end{prop}

\begin{proof}[Proof of Proposition \ref{7  prop dichotomy}]
	It is clear that if $f$ is special, then $f$ is not $c$-accessible. 
	In the following, we show that the non-special property guarantees the $c$-accessibility, 
	and the quantitative estimation of the center paths. 
	This proof is a modification of the one for Anosov cases shown in \cite{GS22}. 
	
	Let $f$ be non-special, and fix $w_0\in\TT^2$ a non-special point. 
	By Remark \ref{2 rmk universal to limit}, 
	there exist two lifts $w^*, w_*\in\RR^2$ of $w_0$ such that $D\pi\big(E_F^c(w^*) \big)\neq D\pi\big(E_F^c(w_*) \big)$. 
	Let $U\subset \TT^2$ be a small neighborhood of $w_0$, 
	and denote its two lifts containing $w^*, w_*$ by $U^*, U_*\subset \RR^2$, respectively. 
	Denote the local center foliations of $U^*$ and $U_*$ by $\mcf^*$ and $\mcf_*$ respectively. 
	By the continuity of center bundle with respect to orbits, 
	the projections $\pi(\mcf^*)$ and $\pi(\mcf_*)$ are transverse in $U$ when $U$ is small enough. 
	Assume that the sizes of the leaf of $\pi(\mcf^*)$ and $\pi(\mcf_*)$ are both smaller than $R_0>0$.
	
	\begin{claim}\label{3  claim c is minimal}
		There exists $R_1>0$ such that for any $z \in\RR^2$,   
		$$\mcf_F^c(z,R_1) \cap \big( U^* +\ZZ^2 \big) \neq \emptyset.$$
		In particular, $\pi\big(  \mcf_F^c(z) \big)$ is dense in $\TT^2$.
	\end{claim}
	
	\begin{proof}[Proof of Claim \ref{3  claim c is minimal}]
		Let $H:\mrr$ be the semi-conjugacy between $F$ and $A$. 
		We claim that there is a constant $R_2>0$ such that for any $y\in\RR^2$,
		\begin{align}
			\tildeL^u(y,R_2)\cap H\big( U^*+\ZZ^2 \big) \neq \emptyset. \label{eq. 4. 1}
		\end{align}
		Indeed, let $T^*$ be a local  leaf of $\mcf_F^s$ contained in $U^*$. 
		By Proposition \ref{2 prop foliation on R2}, 
		$H$ restricted to each  leaf of $\mcf_F^s$ is homeomorphic to a leaf of $\tildeL^s$.  
		Hence $H(T^*)$ has a positive length. 
		It follows from both items of Proposition \ref{2 prop semiconj on stable} that 
		there are a curve $T\subset H(T^*)$ and $k_0\in\NN$ large such that 
		\[T+A^{k_0}\ZZ^2\subset H\big( U^*+A^{k_0}\ZZ^2\big). \]
		Then by Proposition \ref{2 prop minimal foliation}, 
		there exists a constant $R_2>0$ such that 
		$\big(\tildeL^u(y,R_2)+A^{k_0}\ZZ^2\big) \cap T\neq \emptyset$ 
		for every $y\in\RR^2$, and hence \eqref{eq. 4. 1} holds.
		
		In particular, for any $z\in\RR^2$,  
		we apply \eqref{eq. 4. 1} to $y=H(z)$ and get 
		\[ \tildeL^u\big(H(z),R_2\big)   \cap H\big( U^*+\ZZ^2 \big)\neq \emptyset.\]
		Since $H$ maps $\mcf^c_F$ to $\tildeL^u$ and its restriction to stable leaf is homeomorphic, 
		\[H^{-1} \big(   \tildeL^u\big(H(z),R_2\big)  \big) \subset \mcf^c_F(z,R_1),\] 
		where the constant $R_1>0$ is guaranteed by 
		the quasi-isometry of $\mcf^c_F$ and $H$ being bounded from Id$_{\RR^2}$. 
		Therefore we get $\mcf_F^c(z,R_1) \cap \big( U^* +\ZZ^2 \big) \neq \emptyset$, 
		by the homeomorphism $H|_{\mcf^s_F}$ again.		
	\end{proof}
	
	Let $x\in\TT^2$. By Claim \ref{3  claim c is minimal}, 
	there exists a lift $x_0'\in\RR^2$ of $x=x_0$ such that 
	$\pi\big(\mcf_F^c(x_0',R_1)\big)$ intersects with $U$. 
	Let $x_1\in \pi\big(\mcf_F^c(x_0',R_0)\big)\cap U$. 
	Then by the choice of $U$, 
	the local leaves $\pi(\mcf_*)(x_1)$ and $\pi(\mcf^*)(w_0)$ have intersection point $x_2$. 
	Namely, the lift $x'_1\in U_*$ of $x_1$ satisfies $x_2\in \pi\big( \mcf^c_F(x'_1,R_0) \big)\cap \pi(\mcf^*)(w_0)$. 
	Finally, let $x_2'\in U^*$ be a lift of $x_2$.  
	We have
	$x_3=w_0\in \pi\big( \mcf^c_F(x'_2,R_0) \big)$. 
	
	Hence, $R=\max\{R_0,R_1\}$ and $N=3$ are the constants satisfying Proposition \ref{7  prop dichotomy}.
\end{proof}

Then we consider the rigidity of $s$-DA maps. 
Note that the authors get the following result in \cite{GX2023}.
\begin{prop}[\!\cite{GX2023}\,]\label{7 prop special rigidty}
	Let $f:\mtt$ be a $C^{1}$-smooth 
	$s$-DA map with linearization $A:\mtt$.
	Then $f$ is semi-conjugate to $A$ via $h:\mtt$, if and only if, $f$ is special.
	In particular, each of them implies that 
	$\Lambda:=\overline{\big\{x\in\TT^2~ |~  h^{-1}\circ h(x)=\{x\}\big\}}$ is invariant,  
	$\overline{{\rm Per}(f|_{\Lambda})}=\Lambda$, 
	and $\lambda^s(p,f)\equiv\lambda^s(A)$ for any $p\in {\rm Per}(f|_{\Lambda})$. 
\end{prop}

Proposition \ref{7  prop dichotomy} and Proposition \ref{7 prop special rigidty} lead 
a general semi-conjugacy rigidity result of $s$-DA maps on $\TT^2$.

\begin{prop}\label{7 prop dich 2}
	Let $f:\mtt$ be a $C^1$-smooth $s$-DA map. 
	If $f$ is semi-conjugate to a $C^1$-smooth Anosov map $g:\mtt$ via $h:\mtt$, 
	then one has the following dichotomy:
	\begin{itemize}
		\item Either $f$ and $g$ are both special. 
		In this case, $\lambda^s(p,f)=\lambda^s(h(p),g)$ for each $f$-periodic point $p$ 
		on the closure of the set of $h$-injection points.
		
		\item Or $f$ and $g$ are both $c$-accessible. 
		In this case, $h$ is a homeomorphism.
	\end{itemize}
\end{prop}

\begin{proof}[Proof of Proposition \ref{7 prop dich 2}]
	The proof of $f$ and $g$ being special or $c$-accessible simultaneously is 
	the same as the proof of Corollary \ref{2 cor both special}. 
	And the rigidity part of the special case follows from Proposition \ref{7 prop special rigidty}, immediately. 
	Now we assume that $f$ is $c$-accessible and we prove that the semi-conjugacy $h$ is in fact a homeomorphism.
	
	By contradiction, let $h$ be non-injective. 
	Let $F, G$ and $H:\mrr$ be lifts of $f, g$ and $h$, respectively. 
	Note that $H$ is commutative with the deck transformations. 
	It follows that the set of $H$-injective points 
	\[\Gamma=\{ w\in\RR^2\ |\  H^{-1}\circ H(w)=\{w\}\}, \]
	satisfies $\Gamma+\ZZ^2=\Gamma$. 
	Indeed, let $w\in\Gamma$, $n\in\ZZ^2$ and $y\in H^{-1}\circ H(w+n)$. 
	Then $H(y)=H(w+n)=H(w)+n$ and hence $H(y-n)=H(w), y=w+n\in\Gamma$.
	By the assumption that $h$ is not injective, $\RR^2\setminus \Gamma$ is non-empty.
	
	\begin{claim}\label{4 claim 1}
		Let $x\in\RR^2$. If $I=H^{-1}\circ H(x)\subset \mcf^c_F(x)$ is a center curve with a positive length.  
		Then $(I+n)\subset (\RR^2\setminus \Gamma) \cap \mcf^c_F(x+n)$ and $H^{-1}\circ H(I+n)=I+n$, for any $n\in\ZZ^2$.
	\end{claim}
	\begin{proof}[Proof of Claim \ref{4 claim 1}]
		It follows from $\Gamma+\ZZ^2=\Gamma$ that $I\subset \RR^2\setminus \Gamma$ 
		if and only if 
		$(I+n)\subset \RR^2\setminus \Gamma$. 
		Note that $H(I+n)=H(I)+n=H(x)+n=H(x+n)$. 
		Hence $I+n\subset H^{-1}\circ H(x+n)$. 
		By Proposition \ref{2 prop foliation on R2}, $(I+n)\subset H^{-1}\circ H(x+n)\subset \mcf^c_F(x+n)$. 
		Hence we get that $(I+n)\subset(\RR^2\setminus \Gamma) \cap \mcf^c_F(x+n)$.
		
		Now we take $J=H^{-1}\circ H(x+n)\subset  \mcf^c_F(x+n)$. 
		Since $(I+n)\subset J$, $J$ has a positive length and $J\subset (\RR^2\setminus \Gamma)$. 
		Applying the above proof to $(x+n)$, $J$ and $-n$,  
		$(J-n)\subset H^{-1}\circ H(x)=I$. Hence $J=I+n$.
	\end{proof}
	
	For any $x\in\RR^2$, denote $ \mcf^c_F(x)\cap (\RR^2\setminus \Gamma)$ by $I(x)$.  
	Notice that the set $\RR^2\setminus \Gamma$ is dense in $\RR^2$.
	Then $I(x)$ is dense in $\mcf^c_F(x)$. 
	Note that $I(x)$ is the union of countably many closed curves $I_k$ with $H^{-1}\circ H(I_k)=I_k$. 
	Let $x_k\in I_k$ and $n\in\ZZ^2$. 
	Then by Claim \ref{4 claim 1},  
	$I_k+n=H^{-1}\circ H(x_k+n)\subset \mcf^c_F(x_k+n)$. 
	It follows that $E^c_F(y+n)=D\mathcal{T}_nE^c_F(y)$ for every $y\in I(x)$, 
	so for all $y\in \mcf^c_F(x)$ by the denseness of $I(x)$ in $\mcf^c_F(x)$.  
	Hence $\mcf^c_F(x)+n= \mcf^c_F(x+n)$ for any $x\in\RR^2$ and $n\in\ZZ^2$. 
	Therefore, $f$ is special. 
	This contradicts the assumption of $c$-accessibility, 
	and ends the proof of Proposition \ref{7 prop dich 2}.
\end{proof}

\section{DA diffeomorphism on $\mathbb{T}^3$}
\label{sec-app-c}

In this appendix, we show Theorem \ref{thm appendix DA} 
i.e., the semi-conjugacy rigidity of DA diffeomorphisms on $\TT^3$. 
Recall that $\phi:\TT^3\to \TT^3$ is a diffeomorphism 
homotopic to an Anosov diffeomorphism $\psi:\TT^3\to \TT^3$. 
It is known that there is a semi-conjugacy $h:\TT^3\to \TT^3$, 
such that $h\circ \phi=\psi\circ h$ \cite{Franks1969}. 
We assume that they admit partially hyperbolic splittings:
\[T\TT^3=E^s_\phi\oplus E^c_\phi\oplus E^u_\phi\quad {\rm and}\quad  T\TT^3=E^{s}_\psi\oplus E^{c}_\psi\oplus E^u_\psi,\]
where $E^c_\psi$ and $E^s_\psi$ are weak stable and strong stable bundles of $\psi$. 
Denote $E^{cs/cu}_\sigma=E^c_\sigma\oplus E^{s/u}_\sigma$ for $\sigma=\phi, \psi$.
We prove that if the semi-conjugacy $h$ maps the foliation $\mcf_\phi^s$ to $\mcf_\psi^s$, 
i.e., $h(\mcf_\phi^{s}(x))=\mcf_\psi^{s}(h(x))$ for all $x\in\TT^3$, 
then $\phi$ is an Anosov diffeomorphism, and $h$ is a conjugacy smooth along each center manifold.

First, recall the following properties of $DA$ diffeomorphisms on $\TT^3$.
\begin{prop}[\cite{Potrie2015,HS21}]\label{prop appendix DA topo}
	Let $\phi$ and $\psi$ be $C^1$-smooth. Then $\phi$ and $\psi$ are coherent, i.e., there are foliations $\mcf_{\sigma}^{cs/c/cu}$ tangent to bundles $E^{cs/c/cu}_\sigma$, for $\sigma=\phi, \psi$. Moreover, the semi-conjugacy $h:\TT^3\to \TT^3$ satisfies that
	\begin{enumerate}
		\item For $\tau=u,c,cs,cu $, $h$ maps foliation $\mcf^{\tau}_\phi$ to foliation $\mcf^{\tau}_\psi$.
		\item For every $x\in\TT^3$, the restriction $h|_{\mcf^c_\phi}:\mcf^c_\phi(x)\to \mcf^c_\psi(h(x))$ is monotonic.
		\item For every $x\in\TT^3$, the set $h^{-1}\circ h(x)$ is 
		either a single point or a uniformly compact local center leaf.
	\end{enumerate}
\end{prop}

We mainly get a rigidity of the semi-conjugacy for the $su$-accessible case.
\begin{prop}\label{prop appendix DA rigidity}
	Assume that  both $\phi$ and $\psi$ are $C^1$-smooth and $su$-accessible.  
	If $h$ maps the foliation $\mcf_\phi^s$ to $\mcf_\psi^{s}$. 
	Then, $h$ is actually a conjugacy and $\phi$ is an Anosov diffeomorphism. 
	Moreover, if both $\phi$ and $\psi$ are $C^{1+\alpha}$-smooth, the restriction $h|_{\mcf^c_\phi}$ is smooth.
\end{prop}
\begin{proof}[Proof of Proposition \ref{prop appendix DA rigidity}]
	The proof is similar to the proof of Theorem \ref{2 thm rigidity accessible}. 
	\begin{claim}\label{claim appendix DA homeo}
		The semi-conjugacy $h$ is indeed a conjugacy.
	\end{claim}
\begin{proof}[Proof of Claim \ref{claim appendix DA homeo}]
	The proof is similar but simpler than Lemma \ref{4 lem conjugacy}. 
	Let $\Lambda$ be the set of $h$-injection point. 
	It follows from the third item of Proposition \ref{prop appendix DA topo} that $\Lambda$ is nonempty.  
	Since $h$ maps foliations $\mcf^{s/u}_\phi$ to $\mcf^{s/u}_\psi$ 
	and $h^{-1}\circ h(x)$ lies in $\mcf^{c}_\phi(x)$, 
	the set $\Lambda$ is saturated by foliations $\mcf^s_\phi$ and $\mcf^u_\phi$. 
	By $su$-accessibility, $\Lambda=\TT^3$.
\end{proof}

\begin{claim}\label{claim appendix DA Holder}
	The conjugacy $h$ is bi-H\"older continuous.
\end{claim}
\begin{proof}[Proof of Claim \ref{claim appendix DA Holder}]
	The proof is similar to Lemma \ref{4 lem Holder conjugacy}.  Note that: 
	\begin{enumerate}[label=(\roman*)]
		\item The restriction $h|_{\mcf^{s/u}_\phi}:\mcf^{s/u}_\phi(x)\to \mcf^{s/u}_\psi(h(x))$ is 
			uniformly H\"older continuous \cite{KH95}.
		
		\item On each leaf $\mcf^{cs/cu}_{\sigma}(x)$, the holonomy of stable/unstable foliation $\mcf^{s/u}_\sigma$ is 	
			well-define and uniformly H\"older continuous, for $\sigma=\phi, \psi$, 
			since the invariant bundles are H\"older continuous \cite{Pesinbook04}.
		\item For any two points on $\TT^3$, one can connect them 
			by finitely many curves tangent to $E^s_\phi$ or $E^u_\phi$ with uniformly bounded lengths \cite{D2003}.
	\end{enumerate} 
	Since the restriction of $h$ along a center leaf is a monotonic homeomorphism, 
	there is a Lipschitz point $x_0$ of $h_{\mcf^c_\phi}$. 
	Since $h$ now preserves all foliations, the items (ii) and (iii) 
	allow us to send the H\"older continuity of $h_{\mcf^c_\phi}$ at $x_0$ to everywhere with uniform H\"older constants. Combining with the item (i), $h$ is H\"older continuous. 
	By the same argument, $h^{-1}$ is also H\"older continuous.
\end{proof}

From the same proof for the Anosov part of Subsection \ref{subsec C1 case}, we also get that $\phi$ is Anosov here.

The conclusion for the $C^{1+\alpha}$-regularity case follows from the main theorem of \cite{GY24}, 
where the authors showed that
when $\phi$ and $\psi$ are $C^{1+\alpha}$-smooth Anosov and the conjugacy $h$ preserves the strong stable foliation,  
$h$ is $C^{1+\text{H\"older}}$-smooth along each center manifold. 	Here, the H\"older exponent depends on the regularity of center manifolds. This ends the proof of Proposition \ref{prop appendix DA rigidity}.
\end{proof}

\begin{proof}[Proof of Theorem \ref{thm appendix DA}]
	Recall that by Theorem \ref{thm DA diffeo}, 
	$\phi$ is either $su$-accessible, or $su$-integrable.  
	Since $h$ maps $\mcf^{s/u}_\phi$ to $\mcf^{s/u}_\psi$, 
	$\phi$ and $\psi$ are simultaneously $su$-accessible or $su$-integrable. 
	In particular, the conclusion follows from Proposition \ref{prop appendix DA rigidity} for the $su$-accessible case, 
	and from Theorem \ref{thm DA diffeo} for the $su$-integrable case.
\end{proof}

\bibliographystyle{plain}
\bibliography{Bib-DAnoAF}

\end{document}